\documentclass[a4paper,12pt]{article}
\usepackage[T1]{fontenc}
\usepackage[utf8]{inputenc}
\usepackage{lmodern}
\usepackage[english]{babel}

\usepackage{amsmath, amsthm, comment,color}
\usepackage{amsfonts}
\usepackage{amssymb}
\usepackage{graphicx}
\usepackage{hyperref,cleveref}

\theoremstyle{definition} 

\newtheorem{definition}{Definition}[section]
\newtheorem{remark}[definition]{Remark}
\newtheorem{example}[definition]{Example}

\theoremstyle{plain}      

 \newtheorem{Definition}[definition]{Definition}
 \newtheorem{proposition}[definition]{Proposition}

 \newtheorem{theorem}[definition]{Theorem}
 \newtheorem{corollary}[definition]{Corollary}
 \newtheorem{lemma}[definition]{Lemma}

\newtheorem*{theorem*}{Theorem}

\newcommand{\N}{\mathbf N}

\newcommand{\R}{\mathbf R}

\renewcommand{\H}{\mathbf H}
\newcommand{\F}{\mathbf F}

\newcommand{\simplelimit}{\xrightarrow[]{\mathrm{c.s.}}}

\DeclareMathOperator{\Aut}{Aut}
\DeclareMathOperator{\Isom}{Isom}
\DeclareMathOperator{\Cone}{Cone}

\DeclareMathOperator{\CAT}{CAT}
\DeclareMathOperator{\Hdim}{Hdim}
\DeclareMathOperator{\Repr}{Rep}
\newcommand{\RepX}{\Repr_X}
\newcommand{\Schottkyspace}{\Repr^{\mathrm{Schottky}}}
\newcommand{\SchottkyspaceX}{\RepX^{\mathrm{Schottky}}}
\DeclareMathOperator{\diam}{diam}
\newcommand{\omegalim}{\omega-\lim}

\begin{document}

\author{G. Courtois\thanks{IMJ-PRG, Sorbonne Université and Université Paris Cité,
    CNRS, \texttt{gilles.courtois@imj-prg.fr}; work partially funded by ANR-23-CE40-0012-03 \href{https://hilbertxfield.pages.math.cnrs.fr/}{HilbertXField}} \and A. Guilloux\thanks{IMJ-PRG and OURAGAN, Sorbonne Université and Université Paris Cité,
    CNRS, INRIA, \texttt{antonin.guilloux@imj-prg.fr}; work partially funded by ANR-23-CE40-0012-03 \href{https://hilbertxfield.pages.math.cnrs.fr/}{HilbertXField}}}

\title{Hausdorff dimension, diverging Schottky representations and the infinite dimensional hyperbolic space}

\date{}

\maketitle

\begin{abstract}
One of our main goals in this paper is to understand the behavior of limit sets of a diverging sequence of Schottky groups in $\Isom\,\H^N$. This leads us to a generalization of a classical theorem of Bowen on variations of Hausdorff dimension of limit sets; and to a method of transforming a diverging sequence of Schottky groups in $\Isom\,\H^N$ into an almost converging sequence in $\Isom\,\H^\infty$. Our results apply in particular to an example of McMullen and generalize a previous work by Mehmeti and Dang. 
\end{abstract}

\section{Introduction}

One of our main goals in this paper is to understand the behavior of limit sets of a diverging sequence of Schottky groups in  $\Isom\,\H^N$. 

Schottky groups are generated by a finite set
of hyperbolic isometries $S = \{s_1, ... , s_r \}$, $2 \leq r$, of $\Isom\,\H^N$ such that there exists open subsets with disjoint closure
$(D_{i} ^{\pm})_{1\leq i \leq r}$ in $\H^N \cup \partial \H^N$ satisfying 
\begin{equation}\label{schottkydisk}
s_{i}^{\pm} (D_{i}^{\pm}) \subset D_{i}^{\mp}.
\end{equation}
By the classical 'ping-pong' argument, such a group is discrete and free. It may be then considered as the image $\rho (\F_r)$ of the free group on
$r$ generators by a discrete and faithful representation
$\rho : \F _r \rightarrow  \Isom\,\H^N$. 

In this paper, the space of representations of a finitely generated group $\Gamma$ into the isometry group $\Isom\, X$ of a metric space is denoted by $$\RepX(\Gamma).$$ Moreover, we denote by $\Schottkyspace_{\H^N} (\F_r)$ the subspace of $\Repr_{\H^N}(\Gamma)$ such that $\rho (\F_r)$ is a Schottky group. For any
$\rho$ in $\Schottkyspace_{\H^N} (\F_r)$, we denote by
$\Lambda _\rho$ the limit set of $\rho (\F _r)$. The Hausdorff dimension $\Hdim \Lambda_\rho$ of a subset $\Lambda \subset \partial\H ^N$ is computed with 
respect to the visual distance at a base point $o\in \H ^N$ on $\partial \H ^N$, see \Cref{sec:visualmetric}.

The variations of the limit set $\Lambda_\rho$ and its Hausdorff dimension $\Hdim \Lambda_\rho$ ,when $\rho$ varies in 
$\Schottkyspace_{\H^N} (\F_r)$, has been studied thoroughly. 
The following celebrated continuity result is a consequence of the work of Bowen \cite{Bowen-IHES}. 
\begin{theorem}[\cite{Bowen-IHES}]\label{bowen}
Let $\rho _k$, $k\in \N$, and $\rho$ be Schottky representations in the space
$\Schottkyspace_{\H^N} (\F_{r})$. We assume that 
$\lim_k \rho _k (g) =\rho (g)$ for every 
$g\in \F _r$. Then, 
$$
\lim _k \Hdim \Lambda _{\rho _k} = \Hdim \Lambda _\rho .
$$
\end{theorem}

One motivating question for this paper is the following: what happens when the sequence $\rho_k$ is diverging? In the sequel, we study the asymptotic behavior of $\Hdim \Lambda _{\rho _k}$ in the general situation where 
$\rho _k \in \Schottkyspace_{\H^N} (\F_{r})$ diverges, i.e. the minimal joint displacement $r_k$ of the generators goes to $+\infty$ (see Section \ref{diverging} for a precise definition), where:
$$
r_k := \inf _{x\in \H^N} \max _{1\leq j \leq r} d(x, \rho _k (s_j) x).
$$ 
In this case, we expect
that each Schottky set $D_i^{\pm}$ as defined in (\ref{schottkydisk}) tends to a point and 
that the Hausdorff dimension $\Hdim \Lambda _{\rho _k}$ tends to $0$. We want 
to give an asymptotic expansion for the sequence.

This question will lead us to a generalization of Bowen \Cref{bowen}, see \Cref{thm:bowen-CAT(-1)}; and to a method of transforming a diverging sequence of Schottky groups in $\Isom\,\H^N$ into an almost converging sequence in $\Isom\,\H^\infty$, see \Cref{thm:intro-partial-hausdorff-convergence}. 

\subsection{Revisiting an example of McMullen}

Our results apply in particular to an example of McMullen, that can be used as a guideline to our paper.  This example indeed describes such an asymptotic. 
\begin{example}[McMullen example]\label{ex:mcmullen}
In \cite{McM}, C. McMullen studies the family of discrete groups $G _\theta < \Isom\,\H ^2$ generated by three reflections $\sigma_1$, $\sigma_2$ and $\sigma_3$
in three disjoint symmetric circles, each orthogonal to $\partial \H^2$
in an arc of length $0<\theta <2\pi /3$. The index $2$ subgroup $\Gamma _\theta < G_\theta$ of orientation preserving isometries of $G_\theta$ is a Schottky group on two generators 
$\rho _\theta (s_1):= \sigma_1\, \sigma_2$ and $\rho _\theta (s_2):= \sigma_1 \, \sigma_3$.
When $\theta$ goes to $0$, the family of representations $\rho _\theta$ of the free group on two generators such that 
$\rho _\theta (\F _2) = \Gamma _\theta$ is diverging and McMullen shows that 
\begin{equation}\label{McMullenexample}
\Hdim \Lambda _{\rho _{\theta}} \sim \frac{\log 2}{2 |\log \theta |}.
\end{equation}
\end{example}
Drawing a parallel with \Cref{bowen}, we need a limiting representation to describe the asymptotic. A classical construction due to Morgan-Shalen \cite{MorganShalen}, through an ultrafilter, renormalization and asymptotic cones, see \Cref{diverging}, gives that the sequence $(\H ^N, r_{k}^{-1} d)$ converges to an action of $\F_r$ on a minimal tree $(T, d_T)$ in the equivariant Gromov-Hausdorff topology. We denote by $\rho_\infty : \F_r \to \Isom\,T$ the associated representation.

Let us assume, as in \Cref{bowen} for converging sequences, that the limiting representation $\rho_\infty$ is a Schottky representation, with the same definition as for representations into $\Isom\,\H^n$ and the disc $D_i$'s being inside $\partial T$. Let $\Lambda_\infty$ be the limit set of $\rho _{\infty} (\F_r ) \subset \partial T$. The tree is a $\CAT(-1)$-metric space and its boundary $\partial T$ may naturally be equipped with a visual distance $d_p$ (for a given base-point $p\in T$), see \Cref{sec:visualmetric}.
For a subset $\Lambda \subset \partial T$, we still denote by $\Hdim \Lambda$ the Hausdorff dimension of $\Lambda$ with respect to the distance $d_p$. 
The following Theorem is one of the main result in this paper and explains the asymptotic expansion in McMullen example:
\begin{theorem}\label{hausdimasymp}
Let $(\rho _k)_{k\in \N}$ be a sequence of diverging representations in the space
$\Repr_{\H^N} (\F_{r})$ with joint 
displacement $r_k$. Assume that the limit action of $\rho _\infty (\F _\nu)$ on the minimal tree $T$ is Schottky. Then, up to a subsequence,
$$
\lim _k \; r_k \cdot\Hdim \Lambda _{\rho _k} = \Hdim \Lambda _{\rho _\infty}.
$$
\end{theorem}
\begin{remark}
We will give two different interpretations of the expression \emph{up to a subsequence}. Indeed, as already briefly mentioned, the actual construction of $\rho_\infty$ depends on a choice of an ultrafilter $\omega$, see \Cref{sec:asymptotic-cones}. We will prove first that $\Hdim \Lambda _{\rho _\infty}$ is the $\omega$-limit of  $r_k \cdot\Hdim \Lambda _{\rho _k}$, see \Cref{hausdimasymp-ultrafiltre}.
But, then, in \Cref{prop:convergent_subsequences}, we will explain how to construct an actual subsequence with convergence. Still, this subsequence is not very natural because we cannot guarantee that its set of indices belongs to the ultrafilter $\omega$, see \Cref{sec:convergent_subsequences}.
\end{remark}
\begin{remark}
The question of the asymptotic expansion for Hausdorff dimension of limit sets of diverging Schottky groups has already been studied. In particular, the previous Theorem \ref{hausdimasymp} has been proved by V. Mehmeti and N.B. Dang in the case $N=2, 3$, \cite{MD}, under a slightly more stringent assumption on $\rho_\infty$. The proof, though in several points close to ours, goes through the theory of analytic spaces of Berkovich to deal with the convergence, where we will mostly use non locally compact $\CAT(-1)$-spaces, such as the (unique up to isometry) separable infinite dimensional hyperbolic space $\H^\infty$, see \Cref{sec:GromovBusemann}.
\end{remark}
Let us now come back to McMullen example:
\begin{example}[McMullen example, continued]
In the above example $\rho _\theta$ of McMullen, we can compute that the joint displacement of $\rho_\theta$ verifies: $$r_\theta \sim 4 | \log (\theta/2) |.$$
Notice that the limit group $\rho _\infty (\F _2)$ acts on a trivalent tree $T$ with 
quotient a graph with two vertices joined
by three edges $\alpha, \beta, \gamma$ of length $1/2$ where the loops $\gamma .\alpha$ and $\gamma . \beta$
generate $\rho _\infty (\F _2)$, see \Cref{fig:McMullen}. 
\begin{figure}[t]
\centering 
    \includegraphics[width=10cm]{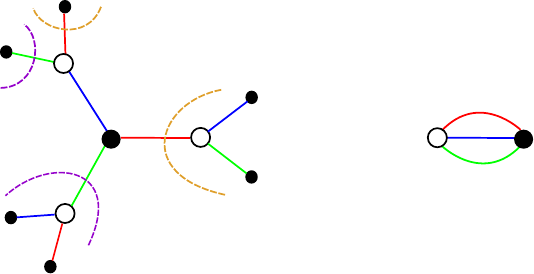}

    \caption{McMullen example: Schottky action of  $\rho_\infty$ on its minimal tree (left); quotient of this action (right).}
    \label{fig:McMullen}
\end{figure}
Therefore, the limit set of 
$\rho _\infty (\F _2)$ coincides with $\partial T$ and its Hausdorff dimension satisfies 
$\Hdim \Lambda _{\rho _\infty} = 2 \log 2$.  Theorem \ref{hausdimasymp}, or more precisely its version \Cref{prop:convergent_subsequences}, states that (see \Cref{ex:mcmullen3}):
$$
\lim _{\theta\to 0} r_\theta \Hdim \Lambda _{\rho _\theta} = 
\lim _{\theta\to 0} 4 |\log (\theta /2)| \Hdim \Lambda _{\rho _{\theta\to 0}} =
\Hdim \Lambda _{\rho _\infty} = 2\log 2
$$
and we recover the estimate \eqref{McMullenexample} of McMullen:
$\Hdim \Lambda _{\rho _{\theta}} \sim_{\theta\to 0} \frac{\log 2}{2 |\log \theta |}$.
\end{example}

We prove \Cref{hausdimasymp} following closely Bowen's method, see \Cref{sec:Bowen-extension}. The main obstacle for doing
so is that the limit sets $\Lambda_\infty$ and $\Lambda_k$'s do not live in the same space, which seems 
to rule out all approximation arguments needed. So our strategy has two steps, which are mainly independent and have each their own interest:
\begin{enumerate}
    \item The first step is to actually transform the Gromov-Hausdorff convergence of $\rho_k$ to $\rho_\infty$ into a weak kind of convergence of a sequence $R_k : \F_r\to \Isom\,\H^\infty$ to $R_\infty: \F_r\to \Isom\,\H^\infty$, see \Cref{intro-embedding} below for more details. Moreover, still denoting by $r_k$ the joint displacement of $\rho_k$, we will have the following formula for the Hausdorff dimensions of limit sets:
    \begin{equation}\label{eq:relations-Hdim}
        r_k\Hdim \Lambda_k = \Hdim \Lambda_{R_k}  \quad \textrm{and} \quad \Hdim \Lambda_\infty  = \Hdim \Lambda_{R_\infty} 
    \end{equation}
    Beware that in these equations, the $\Lambda_k$'s are inside $\partial \H^N$, $\Lambda_\infty$ inside $\partial T$ and all the $\Lambda_{R_k}$ inside $\partial \H^\infty$.
    \item The second step is to state and prove a general version of Bowen that encompass our weak type of convergence. It is best stated in the realm of non locally compact $\CAT(-1)$-spaces, see \Cref{intro-bowen}. It implies the convergence, up to a subsequence: 
    \begin{equation}\label{eq:Bowen-convergence}
        \Hdim \Lambda_{R_k} \to \Hdim \Lambda_{R_\infty}.
    \end{equation}
\end{enumerate}
Eqs. \eqref{eq:relations-Hdim} and \eqref{eq:Bowen-convergence} together give \Cref{hausdimasymp}, see \Cref{sec:hausdimasymp} for more details. We now describe more thoroughly those two steps. We choose to begin with the second step, to reflect the structure of our paper.

\subsection{Bowen Theorem for non locally compact $\CAT(-1)$-spaces}\label{intro-bowen}

\Cref{sec:Bowen-extension} of this paper will be devoted to prove the extension of Bowen's Theorem. We will work in the quite general setting of complete $\CAT(-1)$-spaces without assuming local compactness. One example of such a space is the infinite dimensional separable hyperbolic space $\H^\infty$, see \Cref{H-infini}. Needed facts about $\CAT(-1)$-geometry are reviewed in \Cref{CAT(-1)}.
Matters of convergence, discreteness, convex-cocompactness and related topics for the isometry group of such a space can be delicate \cite{DSU,Duchesne,Xu}. A non trivial notion is the one of \emph{Schottky representations}. We define it in \Cref{def:Schottky} as quasi-isometric representations of $\F_r$ to $\Isom\, X$, i.e. representations $\rho$ such that the orbit map $\tau_\rho: \F_r\to X$, defined by $\tau_\rho(g)=\rho(g)\cdot o$ for all $g\in \F_r$, is a quasi isometric embedding\footnote{The classical notion of $K$-quasi isometric embedding is recalled in \Cref{def:convex_cocompact}.} of the (vertices of the) Cayley graph of $\F_r$ to $X$. The space of such representations is denoted by $\SchottkyspaceX(\F_r)$. 
We prove the following, without assuming $X$ to be locally compact.
\begin{theorem}[Bowen theorem for $\CAT(-1)$-spaces]\label{thm:bowen-CAT(-1)}
Let $(X,d,o)$ be a complete pointed $\CAT(-1)$-space; consider a sequence $(R_k)_{k\in \N}$ in $\RepX(\F_r)$ and $R_\infty\in \SchottkyspaceX(\F_r)$ a Schottky representation.
Assume that the sequence of orbit maps $(\tau_k)_{k\in \N}$ converges point-wise to $\tau_\infty$: 
    $$\forall g\in\F_r,\quad \tau_k(g) = R_k(g) \cdot o \xrightarrow{k \to \infty} \tau_\infty(g) = R_\infty(g) \cdot o.$$   
Then we have:
\begin{enumerate}
    \item The sequence of limit sets $(\Lambda_k)_{k\in \N}$ of $(R_k)_{k\in \N}$ converges, in the Hausdorff topology of $\partial X$, to the limit set $\Lambda_\infty$ of $R_\infty$.
    \item The sequence $\Hdim \Lambda_k$ converges to $\Hdim \Lambda_\infty$.
\end{enumerate}
\end{theorem}

The assumption in the previous Theorem is a consequence of the convergence of $R_k$ to $R_\infty$ in $\RepX(\Gamma)$ (see \Cref{CAT(-1)} for precisions on the compact-open topology of this set). This leads to the following Corollary; it is not surprising, but has not been stated to the best of our knowledge: 
\begin{corollary}\label{coro:BowenAlgConv}
    Let $X$ be a complete $\CAT(-1)$-space, consider $(R _k)_{k\in \N}$ a sequence of representations in $\RepX (\F_{r})$ and $R_\infty$ in $\SchottkyspaceX(\F_{r})$. Assume that 
    $\lim_k R _k (g) =R_\infty (g)$ for every 
    $g\in \F _r$. Then, 
    $$
    \lim _k \Hdim (\Lambda _{R _k}) = \Hdim (\Lambda _{R _\infty}) .
    $$
\end{corollary}
We now present more precisely the construction of the representations $R_k$ a,d $R_\infty$.

\subsection{Turning diverging sequences of representations into convergent ones}\label{intro-embedding}

We work here with a finitely generated $\Gamma$, not especially a free group. When considering diverging sequences of representations of $\Gamma$ into  $\Isom\,\H^N$, it has become classical to consider \emph{asymptotic cones} of sequences of rescaled hyperbolic spaces $X_k = (\H^N, r_k^{-1} d_{\H^N}, o_k)$, where $r_k$ is a sequence of positive real numbers tending to infinity and $o_k\in \H^N$. 
Paulin \cite{Paulin-ENS}, see also \cite{Bestvina}, has explained how the sequence of pointed metric spaces $X_k$ converges in the \emph{equivariant Gromov-Hausdorff} topology to its asymptotic cone, denoted\footnote{We need in fact the choice of an ultrafilter $\omega$, but we will wilfully ignore this technicality in this introduction and refer to \Cref{sec:asymptotic-cones} and \Cref{sec:partial-hausdorff-convergence} for an actually correct version of the story we tell now.} by $X_\omega$.

However, this type of convergence does not allow to apply Theorem \ref{thm:bowen-CAT(-1)}. In \Cref{sec:partial-hausdorff-convergence}, we examine how to promote this equivariant Gromov-Hausdorff convergence to one in the realm of the Theorem. When looking for a better type of convergence, one could hope to get an actual Hausdorff convergence as subsets of a fixed metric space: is it possible to embed isometrically and equivariantly each $X_k$ and $X_\omega$ in a metric space $Y$ such that $X_k$ goes to $X_\omega$ in the Hausdorff topology of subsets of $Y$? It is indeed the case for compact spaces, cf \cite[Section 7.3]{BuragoIvanov}

Note that $X_\omega$ is not locally compact, so this would imply that $Y$ is not locally compact either. Natural candidates for $Y$ are, as such, infinite dimensional hyperbolic spaces. Indeed, nice equivariant embeddings of hyperbolic spaces \cite{MonodPy,Monod-Py-2019} and trees \cite{DSU,Monod-Py-2019} into infinite dimensional hyperbolic spaces have been defined and studied. We will use these embeddings. However, the answer to this crude question is \emph{no}! 
\begin{proposition}\label{nohausdorffconv}
    There is no metric space $Y$, and sequences of isometric embeddings $f_k: X_k \to Y$ and $f:X_\omega \to Y$ such that:
    \begin{itemize}
        \item The limit of any sequence $f_k(x_k)$ equals the image by $f$ of the point $[(x_i)]$ in $X_\omega$;
        \item $f_i(X_i)$ converges, in the Hausdorff topology of $Y$, to $f(X_\omega)$. 
    \end{itemize}
\end{proposition}
\begin{proof}
Indeed, one can look at the unit spheres in $X_k$ and $X_\omega$: the fact that $X_\omega$ is a complete real tree implies that, for any $0<\epsilon <2$, any ball of radius $\epsilon$ in the unit sphere is its own $\epsilon$-neighborhood.

On the contrary, in any $X_k$, the sphere is $\epsilon$-connected: for any pair of points, there is a chain of points connecting them, with each steps of length $<\epsilon$.

This contradicts Hausdorff convergence of $f_k(X_k)$ to $f(X_\omega)$ in $Y$: the limit of a sequence of $\epsilon$-connected subsets should be included in a single ball of radius $\epsilon$ in the image of the unit sphere in $f(X_\omega)$
\end{proof}

So we are looking for an intermediate notion of convergence. In order to explain it, let us come back to \cite{Monod-Py-2019}. Monod and Py studied the self-representations of the infinite dimensional Möbius group, based on the notion of \emph{kernel of hyperbolic type} introduced by Gromov. They have in particular shown, building on their previous \cite{MonodPy} and other contributions \cite{BIM,DSU}, that there are:
\begin{itemize}
    \item a one-parameter family of continuous embeddings $f_t : \H^N \rightarrow \H^\infty$, for all $0\leq t \leq 1$, which are equivariant for a family of representations $\nu_t: \Isom\,\H^N\to \Isom\,\H^\infty$. Denoting by $d$ both the distance in $\H^N$ and $\H^{\infty}$, they verify for all $x,y\in\H^N$:
        \begin{equation*}
            (\cosh {d(x,y)})^t  = \cosh {d(f_t(x), f_t(y))}
        \end{equation*}
   They are very close to being simple dilatations by a factor $t$ and are anyway $1$-Lipschitz, see \Cref{QI1}. They naturally extend to embeddings of $\partial \H^N$ to $\partial \H^\infty$.
    The functions $f_t$ are well-defined up to post-composition by an isometry of $\H^\infty$ as well as the representations $\nu_t$ up to conjugation by the same isometry. We fix one such choice.
    \item for any separable tree $T$, a one-parameter family\footnote{
        This family depends in fact on the tree $T$, we omit this dependence in the notations.
    } of continuous embeddings $F_s : T \rightarrow \H^\infty$, for all $0\leq s$, which are equivariant for a family of representations $\mu_s: \Isom\, T\to \Isom\,\H^\infty$. They are very close to being simple dilatations by a factor $s$. Here also, the functions $F_s$ and representations $\mu_s$ are well-defined, up to an isometry of $\H^\infty$ and we fix one choice. 
\end{itemize}
Those constructions are briefly reviewed in Section \ref{sec:EmbeddingKernel}.

With these families, we can explain our manipulation on the representations $\rho_k$, with joint displacement $r_k$, and $\rho_\infty$. Indeed, let $\rho_k$ be a diverging sequence of representations, $T$ be the minimal invariant tree in the asymptotic cone, which is separable, and $\rho_\infty$ be the limiting representation. Define $R_\infty$ to be the composition $\mu_1\circ \rho_\infty$. It acts on $F_1(T)$ which is an almost isometric embedding of $T$ in $\H^\infty$. In particular, $F_1$ defines a Lipschitz embedding of $\partial T$ to $\partial \H^\infty$.

The asymptotic cone $X_\omega$ and the minimal tree $T$ are obtained by renormalizing the metric space $\H^N$ by the factor $r_k^{-1}$.
The main observation of our paper is that this renormalization can be actually almost isometrically embedded 
into $\H^\infty$, by considering $f_{r_k^{-1}} (\H^N)$, for all $k$ such that $r_k>1$. It 
is not hard to see that the sequence of actions of $\Gamma$ on $f_{r_k^{-1}} (\H^N)\subset \H^\infty$,
through $\nu_{r_k^{-1}}\circ \rho_k$, converges in the equivariant Gromov-Hausdorff sense to the action of $\Gamma$ on $F_1(T)\subset \H^\infty$, through $R_\infty$. The content of our third Theorem is that 
we can choose a sequence $\mathcal I_k$ of isometries of $\H^\infty$, defining the representations and the $R_k$-equivariant embeddings $\phi_k$ by:
$$R_k:= \mathcal I_k\circ (\nu_{r_k^{-1}}\circ \rho_k)\circ \mathcal I_k \textrm{ and } \phi_k:=\mathcal I_k\circ f_{r_k^{-1}} : \H^N\to\H^\infty \textrm{ with }\phi_k \circ \rho _k = R_k \circ \phi _k,$$
such that the convergence of $\phi_k(X_k)$ to $F_1(T)$
is much better and close to an actual Hausdorff convergence. We call this new convergence a 
\emph{partial equivariant Hausdorff convergence}. Its actual definition needs the choice of an ultrafilter; informally, equivariant partial Hausdorff convergence is defined in the following way (see \Cref{def:partial-Hausdorff-convergence}):

\noindent We say that a sequence of embeddings $(\phi_k)_{k\in \N}: \H^N\to \H^\infty$ realizes a \emph{partial equivariant Hausdorff convergence} of $\H^N$ to $F_1(T)$ if one can construct
a sequence of finite subsets $(K_k)_{k\in \N}$  of $\H^N$
such that $F_1 (T)$ is the set of accumulation points of sequences $(\phi_k(x_k))$, where $x_k\in K_k$ and the actions of $\Gamma$ on $F_1(T)$ through $R_\infty$ and on sequences through the sequence $(R_k)$ are compatible.

Our third main Theorem completes the first step of the strategy for \Cref{hausdimasymp}:
\begin{theorem}[Equivariant partial Hausdorff convergence, see Theorem \ref{thm:partial-hausdorff-convergence}]\label{thm:intro-partial-hausdorff-convergence}
    For any diverging sequence $\rho_k:\Gamma\to \Isom\,\H^N$, with the notations above, one can choose a sequence $\mathcal I_k$ of isometries of $\H^\infty$ such that the sequence of embeddings $(\phi_k:=\mathcal I_k\circ f_{r_k^{-1}})_{k\in \N}$ realizes a partial equivariant Hausdorff convergence of $\H^N$ to $F_1(T)$. 
    
    Moreover, the representations $R_k$ preserve $\mathcal I_k(\H^N)$ and the sequence of orbit maps $(\tau_{R_k})_{k\in \N}$ "converges" point-wise to $\tau_{R_\infty}$.   
\end{theorem}
Above, the notion of "convergence" is actually a convergence along the ultrafilter, see  \Cref{thm:partial-hausdorff-convergence} for the precise statement. Surprisingly, a variation of our theorem is given in \Cref{sec:convergent_subsequences}, where explicit subsequences are build along which the representations $R_k$ actually converge toward $R_\infty$ in $\Repr_{\H^\infty}(\Gamma)$. In particular, it shows that our construction for McMullen example leads to an actual convergence of $R_\theta$ to $R_\infty$ in $\Repr_{\H^\infty}(\Gamma)$, see \Cref{ex:mcmullen3}.

The conclusion of \Cref{thm:intro-partial-hausdorff-convergence} is what we need to apply our  \Cref{thm:bowen-CAT(-1)} to the sequence $R_k$. As, moreover, $\phi_k:=\mathcal I_k\circ f_{r_k^{-1}}$ is almost homothetic by a factor $r_k^{-1}$, the induced map between limit sets $\Lambda_{\rho_k}$ and $\Lambda_{R_k}:=\phi_k(\Lambda_{\rho_k})$ acts as an exponentiation to the power $r_k^{-1}$ on the visual distances, see \Cref{sec:hausdimasymp}.
Hence the Hausdorff dimensions verify Eq. \eqref{eq:relations-Hdim}:
$$r_k\Hdim \Lambda_{\rho_k} = \Hdim \Lambda_{R_k}.$$

This proves \Cref{hausdimasymp} and concludes the paper.

\subsection{Acknowledgments}

The authors express their deep gratitude to Nguyen-Bac Dang, Vlerë Mehmeti, Frédéric Naud, Anne Parreau, Frédéric Paulin, Pierre Py, Samuel Tapie, Teddy Weisman, Pierre Will, David Xu and more generally all the members of ANR-23-CE40-0012-03 project \href{https://hilbertxfield.pages.math.cnrs.fr/}{HilbertXField} for fruitful discussions and their enlightning comments.

\tableofcontents

\section{CAT(-1)-geometry}\label{CAT(-1)}

We review in this section standard notations and known facts about $\CAT(-1)$-geometry and prove a few results we will need later on. For convenience, our main reference will be \cite{DSU}; see also \cite{BridsonHaefliger}. The main examples to keep in mind are $\H^\infty$,  metric trees and more generally what \cite{DSU} calls a ROSSONCT.

\subsection{$\CAT(-1)$-spaces}\label{sec:GromovBusemann}

A $\CAT(-1)$-space $(X,d)$ is a geodesic metric space which satisfy the $\CAT(-1)$-inequality \cite[Eq. (3.2.1)]{DSU} expressing that triangles are thinner than in the hyperbolic plane. We will \emph{always} assume our $\CAT(-1)$-space to be complete, but not necessarily to be locally compact (or proper). 
Such a space admits an ideal boundary $\partial X$, see\footnote{Beware that the lack of properness requires it to be defined as a quotient of a space of sequences rather than geodesic rays.} \cite[Section 3.4]{DSU}.  We will not precisely define it here, but still define two important functions on a $\CAT(-1)$-space: Gromov product and Busemann functions.
\begin{definition}\label{def:GromovBusemann}
Let $X$ be a $\CAT(-1)$-space, for three points $x,y,z$ in $X$, the Gromov product of $x$ and $y$ w.r.t. $z$ is:
$$\langle x,y\rangle_z:= \frac 12 \left( d(x,z)+d(y,z)-d(x,y) \right);$$
the Busemann function is: $x,y,z\mapsto B(x,y,z)=d(x,z)-d(y,z)$. 
\end{definition}
Note that \cite{DSU} denotes this last function by $B_z(x,y)$. Those functions are invariant under the action of $\Isom\, X$ on all variables simultaneously. In the language of hyperbolic spaces of \cite[Section 3.3]{DSU}, $\CAT(-1)$-spaces are \emph{strongly}\footnote{When referring to \cite{DSU}, we state directly the version of their statements holding for strongly hyperbolic spaces, not the weaker general version.} hyperbolic because the Gromov product verify, for all $x,y,z,w\in X$:
\begin{equation}\label{eq:visual_metric_triangular_<}
    e^{-\langle x,z \rangle_w} \leq e^{-\langle x,y \rangle_w} + e^{-\langle y,z \rangle_w}
\end{equation}

The Gromov product $\langle x,y\rangle_z$ extends continuously \cite[Lemma 3.4.22]{DSU} to $x\neq y\in \partial X$ and $z\in X$; the Busemann function $B(x,y,z)$ to $x,y\in X$, $z\in \partial X$. They furthermore verify a series of relations stated in \cite[Prop 3.3.3]{DSU} that we do not recall here, but will use with proper reference when needed. In particular, the functions $x\to B(x,y,\xi)$ and $y\to B(x,y,\xi)$ are Lipschitz functions on $(X,d)$; the function $\xi \to B(x,y,\xi)$ is Lipschitz with respect to the visual metric $d_o$ (defined in the next session), see also \cite{Bo}. The invariance of the extended Busemann functions is given, for every $x,y \in X$, $\zeta \in \partial X$ and $g\in \Isom\, X$, by:
\begin{equation}\label{busemaninvariance}
B(x,y,\zeta) = B(gx,gy,g\zeta).
\end{equation}

As we have assumed $X$ to be complete, it is regular hyperbolic, see \cite[Prop. 4.4.4]{DSU}. In particular, between any two distinct points in $X\cup\partial X$, there is a unique geodesic and this geodesic varies continuously w.r.t. the two points.

\begin{example}[Infinite dimensional hyperbolic spaces]\label{H-infini}
Fix an infinite cardinal $\alpha$. Let $H_{\alpha}, b_\alpha$ be a real Hilbert space of cardinality $\alpha$. We consider the Minkowski space $\R\oplus H_\alpha$ with the bilinear form $(s,u),(s',u') \mapsto ss' - b_\alpha(u,u')_\alpha$.
The infinite dimensional hyperbolic space $\H^\alpha$ of cardinal $\alpha$ is defined as:
$$\H^\alpha = \left\{ (s,u)\in \R\oplus H_\alpha \textrm{ s.t. }s^2-(u,u)_{H_\alpha} = 1, \, s>0\right\}.$$
Its boundary $\partial \H^\alpha$ is canonically identified to $\{(1,u)\}$ where $u$ belongs to the unit sphere $b_\alpha(u,u)=1$ in $H_\alpha$.

If $\alpha$ is countable, we will denote this space instead by $\H^\infty$. We will actually need only this separable hyperbolic space in this paper.
\end{example}

\subsection{Visual metric and shadows}\label{sec:visualmetric}

Let $(X,d,o)$ be a pointed $\CAT(-1)$-space. Then one defines the visual metric, from $o$, on $\partial X$ by:
$$d_o(\xi,\zeta) = e^{-\langle \xi, \zeta\rangle_o}.$$
The triangular inequality is granted by Eq. \eqref{eq:visual_metric_triangular_<}. It is complete on $\partial X$, see \cite[Observation 3.6.7]{DSU}.

Given a set $E\subset X$, one defines its shadow $\pi^{\partial X}_o(E)$ as the set of points $\xi$ in $\partial X$ such that the geodesic between $o$ and $\xi$ meets $E$. A very useful estimates for our purposes is the fact that the shadow of a moderate size ball $\mathcal B(z,\sigma)$ centered at $z$ far away from $o$ has a small diameter for $d_o$. Let us denote, for any $E'\subset \partial X$, by $\diam_o(E')$ its diameter for $d_o$.
\begin{proposition}\label{prop:diam_shadows}
Let $(X,d,o)$ be a pointed complete $\CAT(-1)$-space. For each $r>0$, there exists $c_r>0$ such that for all $z\in X$ we have:
$$\diam_o(\pi^{\partial X}_o(\mathcal B(z,r))) \leq c_r e^{-d(o,z)}.$$
\end{proposition}
\begin{proof}
This is a combination of Corollary 4.5.5 and Lemma 4.5.8 of \cite{DSU}.
\end{proof}

\subsection{Quasi-isometric representations}\label{sec:convex-cocompact}

The group $\Isom\, X$ of isometries of $X$ is equipped with its compact-open topology, see \cite[Section 5.1]{DSU} and \cite{Duchesne}. This topology amounts to the pointwise convergence: $g_n\to g$ iff for all $x\in X$, $g_n(x)\to g(x)$.
Any transformation $g\in\Isom\, X$  extends continuously as a Lipschitz transformation of $(\partial X,d_o)$. 

We fix now a finitely generated group $\Gamma$, with a generating set $S$ and 
hence a $\Gamma$-invariant length metric on $\Gamma$: $|\gamma|_\Gamma$ equals the minimal length of
an expression of $\gamma$ as a product of generators. It is the length of paths in the Cayley graph of
$(\Gamma,S)$. A \emph{geodesic ray} in $\Gamma$ is a sequence $\gamma_0=e,\,\gamma_1,\ldots$ such that
$|\gamma_i|_\Gamma=i$ for all $i$.

The space $\RepX(\Gamma)$ is endowed with the compact-open topology, which once again amounts to pointwise convergence: $\rho_n\to\rho_\infty$
iff for all $\gamma\in\Gamma$ we have $\rho_n(\gamma)\to\rho_\infty(\gamma)$ in $\Isom\, X$.
A crucial notion for our study of representations is the orbit map:
\begin{definition}[Orbit map]
For any $\rho \in \RepX(\Gamma)$, the \emph{orbit map} of $\rho$ is the map:
\begin{align*}
    \tau_\rho:\: & \Gamma \to X\\
    & \gamma \mapsto \rho(\gamma)\cdot o
\end{align*}
\end{definition}
We will repeatedly encounter the situation where a sequence of orbit maps $\tau_k$ converge simply to an orbit map $\tau_\infty$, i.e. for all $\gamma\in\Gamma$ we have $\tau_k(\gamma) \to \tau_\infty(\gamma)$. We will denote this situation by:
$$\tau_k \simplelimit \tau_\infty.$$

The notion of convex-cocompact representation is subtle in $\CAT(-1)$-spaces. In fact, we restrict to the related notion of \emph{quasi-isometric representation}. Recall \cite[Definition 3.3.9]{DSU} that a map $f : (A,d_A) \to (B,d_B)$ between two metric
spaces is a $(K,C)$-quasi isometric embedding (for $K\geq1$, $C\geq 0$) if we have for all $a,a'$ in $A$:
$$\frac 1K d_A(a,a')-C\leq d_B(f(a),f(b))\leq K d_A(a,a')+C.$$
When $C=K$ in the above definition, we will simply say $K$-quasi isometric embedding.
\begin{definition}[Quasi-isometric representation]\label{def:convex_cocompact}
Let $\Gamma$ be a finitely generated group and $(X,d,o)$ be a pointed complete $\CAT(-1)$-space. 
A representation $\rho \in \RepX(\Gamma)$ is $K$-quasi-isometric (for a $K\geq 1$) if 
the orbit map $\tau_\rho$
is a $K$-quasi isometric embedding of $(\Gamma,|\cdot|_\Gamma)$ into $(X,d)$.
\end{definition}
We will say simply \emph{quasi-isometric representation} when the constant $K$ is not important. When $K$ is not important, the previous definition does not depends on the choice of the point $o$: if we change the origin to $o'$, the constant $K$ changes, but not the quasi-isometric property for the orbit map associated to $o'$.
For $X$ locally compact or $X=\H^\infty$, a representation $\rho$ is quasi-isometric iff it is convex-cocompact, i.e. there exists a $\rho(\Gamma)$-invariant convex subset of $X$ on which $\Gamma$ acts cocompactly, see \cite[Theorem 1.1]{Xu}. We do not know if the same holds for a general $X$.

For a group $\Gamma$ to admit quasi-isometric representations in a $\CAT(-1)$ space, 
it must at least be hyperbolic \cite[Theorem III.H.1.9]{BridsonHaefliger}. As such, it admits a visual boundary $\partial \Gamma$. For every quasi-isometric representation $\rho$, the orbit map $\tau_\rho$ extends continuously to a map, still denoted
$\tau_\rho$, from $\partial \Gamma$ to $\partial X$. For such a group, any geodesic ray in $\Gamma$ has a well-defined endpoint $\xi \in \partial \Gamma$.

An important result for our purpose is that the set of quasi-isometric representations is open in $\RepX(\Gamma)$, and even open for the coarser topology given by the simple convergence of orbit maps:
\begin{theorem}[{\cite[Theorem 3.8]{Xu}}]\label{thm:quasi_isometric_open}
Let $(X,d,o)$ be a pointed complete $\CAT(-1)$-space and $\Gamma$ be a finitely generated group. 
If $\rho_\infty : \Gamma \to \Isom\, X$ is a 
$K$-quasi-isometric representation, then for some $K'>K$, for every sequence $(\rho_k)$ in $\RepX(\Gamma)$ such that $\tau_{\rho_k}\simplelimit\tau_{\rho_\infty}$ all but a finite number of $\rho_k$'s are $K'$-quasi-isometric.
\end{theorem}
\cite[Theorem 3.8]{Xu} is only stated for the compact-open topology, but its proof actually gives the previous result. Let us briefly sketch this proof.
\begin{proof}
The key point is a local-to-global principle for quasi-isometric embeddings \cite[Theorem 3.1.4]{CDP}: there exists constants $K',A$ such that for all map $\tau:\Gamma\to X$, if $\tau$ is a $(K+1)$-quasi-isometric embedding on each ball of radius $A$ in $\Gamma$, then $\tau$ is a $K'$-quasi-isometric embedding of $\Gamma$ into $X$. If $\tau$ is moreover equivariant by a representation $\rho:\Gamma\to\Isom\, X$, it is enough to check that $\tau$ is a $(K+1)$-quasi-isometric embedding of the ball in $\Gamma$ centered at the identity and of radius $A$ into $X$.

Now, this ball is finite, and by assumption $\tau_{\rho_\infty}$ is a $K$-isometric embedding of this ball into $X$. As $\tau_{\rho_k}\simplelimit \rho_\infty$, then for $k$ big enough, $\tau_{\rho_k}$ is a $(K+1)$-isometric embedding of this ball into $X$.
By the local-to-global principle, $\tau_{\rho_k}$ is a $K'$-quasi-isometric embedding.
\end{proof}

We will use the fact that images of geodesic rays in $\Gamma$ under a quasi-isometric embedding in $X$ are at quantitatively bounded distance from an actual geodesic in $X$. We tailor the following statement for our needs, but there is a more general version, see \cite[Theorem III.H.1.7]{BridsonHaefliger}.
\begin{theorem}[Stability of geodesics]\label{stability_geodesic}
For each $K>1$, there exists a $C_K>0$ such that for every complete pointed $\CAT(-1)$-space $(X,d,o)$, every $K$-quasi-isometric representation $\rho:\Gamma\to\Isom\, X$, every geodesic ray $(\gamma_i)_{i\geq 0}\subset\Gamma$ with endpoint $\xi\in\partial \Gamma$, we have:
\begin{itemize}
    \item every point in the sequence $\tau_\rho(\gamma_i)$ is at distance at most $C_K$ of the geodesic ray $[o,\tau_\rho(\xi))$ in $X$.
    \item for each $n>0$, every point in the sequence $(\tau_\rho(\gamma_i))_{0\leq i\leq n}$ is at distance at most $C_K$ of the geodesic $[o,\tau_\rho(\gamma_n)]$ in $X$.
\end{itemize}
\end{theorem}

We will need in several places approximation results for limit sets. The previous result translates into such a useful estimate which approximates a point $\tau_\rho(\xi) = \lim \tau_\rho(\gamma_i)$ by an endpoint $\zeta$ of a geodesic ray in $X$ going
through the point $\tau_\rho(\gamma_i)$ of the orbit of $o$:
\begin{proposition}\label{approxlimitset2}
For each $K>1$, for every complete pointed $\CAT(-1)$-space $(X,d,o)$, every $K$-quasi-isometric representation $\rho:\Gamma\to\Isom\, X$, every geodesic ray $(\gamma_i)_{i\geq 0}\subset\Gamma$ with endpoint $\xi\in\partial \Gamma$, we have
for any $i\geq 0$, for any $\zeta \in \partial X$ such that $\tau_\rho(\gamma_i)\in [o,\zeta)$:
$$d_o(\zeta,\tau_\rho(\xi))\leq c_{C_K} e^{d(o,\tau_\rho(\gamma_i))} \leq  c_{C_K}e^{K}e^{-\frac iK}.$$
\end{proposition}
\begin{proof}
From \Cref{stability_geodesic}, we know that the geodesic ray $[o,\tau_\rho(\xi))\subset X$ meets the 
ball of radius $C_K$ around
$\tau_\rho(\gamma_i)$. In other terms, both $\tau_\rho(\xi)$ and $\zeta$ belong to 
$\pi_o^{\partial X}(\mathcal B(\tau_\rho(\gamma_i),C_K))$.

Moreover, from the $K$-quasi-isometric property for $\rho$, we deduce that 
$$d(o,\tau_\rho(\gamma_i))\geq \frac 1K |\gamma|_\Gamma - K\geq \frac iK - K.$$
We conclude the proof using \Cref{prop:diam_shadows}, with $r = C_K$:
\begin{equation*}
    d_o(\tau_\rho(\xi),\zeta) \leq \diam_o(\pi_o^{\partial X}(\mathcal B(\tau_\rho(\gamma_i)),C_K))
    \leq  e^{-d(o,\tau_\rho(\gamma_i))}
    \leq c_{C_K}e^{K}e^{-\frac iK}.
\end{equation*}
\end{proof}

This ends our review of needed facts about $\CAT(-1)$-geometry and quasi-isometric representations. We can now tackle the proof of the extension of Bowen theorem.

\section{Convergence of Hausdorff dimensions for Schottky actions in CAT(-1)-spaces}\label{sec:Bowen-extension}

Throughout this section, we consider a pointed complete $\CAT(-1)$ metric space $(X,d,o)$, as 
presented in the previous section. Recall that we do \emph{not} assume it is locally compact. 
With respect to the previous section, we restrict our attention to $\Gamma = \F_r$, the free group 
over $r$ generators $s_1,\ldots,s_r$. Let $\mathrm{Cay}_r$ be the Cayley graph of $\F_r$, 
a complete $2r$-tree with edges of length $1$. 
We define, using the notion of quasi-isometric representation (\Cref{def:convex_cocompact}):
\begin{definition}[Schottky groups in $\Isom\,  X$]\label{def:Schottky}
A $K$-quasi-isometric representation $\rho : \F_r\to \Isom\, X$ is called $K$-Schottky.

A subgroup of $\Isom\, X$ is called \emph{Schottky} if it is the image of such a representation.
\end{definition}
If there are open subsets $D_i^\pm$ in $\partial X$ \emph{whose convex hulls in $X\cup \partial X$ have disjoint closures} such that each generator $s_i$ sends the exterior of $D_i^-$ inside $D_i^+$ and its inverse sends the exterior of $D_i^+$ inside $D_i^-$, then one can construct such a quasi-isometric embedding $\mathrm{Cay}_r\to X$ by looking at the orbit of any point outside all convex hulls of $D_i$. It is not clear for us if for any complete $\CAT(-1)$ space $X$ and any Schottky subgroup, such $D_i^\pm$ always exist in $\partial X$. From Maskit theorem \cite{Maskit}, this is true for Kleinian groups in $\Isom\,\H^N$. Moreover, in the special case of $X=\Isom\,\H^\infty$, Xu \cite[Theorem 1.1]{Xu}, proves that our definition is equivalent to the fact that the convex hull of the limit set $\Lambda_\rho$ is locally compact and acted on cocompactly by $\rho(\F_r)$.

The proof of \Cref{{thm:bowen-CAT(-1)}} follows closely the original proof of Bowen, substituting $\CAT(-1)$ arguments to original conformal estimates.

\subsection{Gibbs measures and coding of the limit set}\label{sec:bowen}

Let us first recall the following classical Lemma that we will use to estimate Hausdorff dimensions. For $(\Lambda,d)$ a metric space, we denote by $B_\Lambda(x,r)$ the ball in $\Lambda$ of radius $r$ centered at $x$.
\begin{lemma}[{\cite[Proposition 4.9]{Fa}}]\label{hausdorffmeasure}
Let $(\Lambda, d)$ be a compact metric space. Assume that there exist positive constants
$\delta, c, C$ and a finite mass Borel measure $m$ on $\Lambda$ such that for every $x\in \Lambda$ and $0<r<\epsilon$ we have: $$cr^\delta \leq m(B(x,r)) \leq C r^\delta.$$

Then we have $\Hdim \Lambda = \delta$.
\end{lemma}
In order to compute the Hausdorff dimension of the limit set of a Schottky group, we will find
a measure as in Lemma \ref{hausdorffmeasure}. 

\subsubsection{Shift and cylinders in $\partial \F_r$}

The boundary $\partial \F_r$ of the free group 
$\F _r$ over the set of $r$ generators $S=\{s_1,\ldots, s_r\}$ is the set of infinite reduced words in the generators:
$$\partial \F_r = \{\sigma_1 \sigma _2...\sigma _n..... \,| \, \sigma_i \in S\, ,\, \sigma_i \sigma _{i+1} \neq 1\}.$$
Let $d_1$ be the ultrametric distance on $\partial \F_r$ defined for $\zeta = (\sigma _i)_{i\in \N _{>0}}$ and
$\zeta' = (\sigma _i ')_{i\in \N _{>0}}$ by:
\begin{equation*}
d_1(\zeta, \zeta ') := e^{-n},
\end{equation*}
where $n$ is the largest
integer such that $\sigma _i = \sigma _i '$ for $i = 1,...,n$. It coincides with the visual distance on the boundary of the tree $\mathrm{Cay}_r$ as a $\CAT(-1)$-space, see \Cref{sec:visualmetric}. We will repeatedly use the geodesic ray in $\F_r$ from the identity to a boundary point $\zeta\in\partial\F_r$, we define a notation for it:
\begin{definition}\label{def:g_k_zeta}
Let $\zeta = (\sigma_1,\ldots)\in\partial \F_r$. For $k\geq 1$, we denote the points in the geodesic ray $[e,\zeta)$ by 
$$g_k(\zeta) := \sigma_1\cdots \sigma_k.$$
Moreover, for $\sigma _1, \ldots, \sigma _n$, we define the \emph{cylinder} $C(\sigma_1 \cdots \sigma_n)$ by:
$$C(\sigma_1 \cdots \sigma_n):= \{ \zeta\in\partial\F_r \,| \, g_n(\zeta)=\sigma_1\cdots\sigma_n\}.$$
\end{definition}
We can reinterpret the distance $d_1$: for two points $\xi,\zeta\in\partial \F_r$, we have
\begin{equation*}
d_1(\xi,\zeta)=e^{-n}\textrm{ where }n\textrm{ is the biggest integer with } g_n(\xi)=g_n(\zeta).
\end{equation*}
A cylinder $C(\sigma_1 \cdots \sigma_n)$ is the set of $\zeta$ such that the geodesic $[o,\zeta)$ goes through $\sigma_1,\cdots,\sigma_n$, or the shadow $\pi_o^{\partial\F_r}(\{\sigma_1 \cdots \sigma_n\})$ cast by the singleton $\sigma_1 \cdots \sigma_n$.
It is an open and closed subset of $\partial \F_r$ and may also be described as the ball, for $d_1$, of radius $e^{-n}$ and center any of its point.

A natural dynamical system on $\partial \F_r$ is the shift:
$$\begin{matrix}
    \mathcal{S} : &\partial \F_r &\to& \partial \F_r\\
    &\zeta = (\sigma _1 , \sigma_2,\ldots)&\mapsto& \mathcal S \zeta = (\sigma_2,\ldots).
\end{matrix}$$ 
It is expanding, by a factor $e^1$, around any point.

\subsubsection{Schottky actions and coding of the limit set}\label{coding}

Let us consider a $K$-Schottky representation $\rho: \F_r \to \Isom\, X$
where $(X,d,o)$ is a complete pointed $\CAT (-1)$-space, i.e. the orbit map $\tau_\rho$ is a $K$-quasi isometric embedding, see \Cref{def:convex_cocompact}. The extension of $\tau_\rho$ to the boundary is given explicitly in this case:
$$\zeta\in\partial\F_r \mapsto  \tau_\rho (\zeta) := \lim_k \tau_\rho(g_k(\zeta))\in\partial X.$$
The limit set $\Lambda _\rho \subset \partial X$ is the image $\tau_\rho(\partial \F_r)$.

The ideal boundary 
$\partial X$ is equipped with the visual distances $d_{o}$, defined with Gromov products see \Cref{sec:visualmetric}. 
Note that if we change the origin $o$ in $o'$, the two distances $d_o$ and $d_{o'}$ are equivalent, so give the same Hausdorff dimension. We consider the Hausdorff dimension of the limit sets $\Lambda _\rho$ with respect to these distances. 

The coding given by the map $\tau_\rho$ of the limit sets of Schottky groups is essential for the continuity of the Hausdorff dimension. The following Lemma may be seen as a variation on  \Cref{approxlimitset2}. Recall the constant $C_K$ defined by \Cref{stability_geodesic}.
\begin{lemma}\label{homeoH}
For every $K$-Schottky representation
$\rho : \F_r \to \Isom\, X$, for all $\xi$ and $\zeta$ in $\partial\F_r$ with $d_1(\xi,\zeta)=e^{-n}$, we have:
$$e^{-3C_K} e^{-d(o,\tau_\rho(g_n(\zeta)))}\leq d_o(\tau_\rho(\xi),\tau_\rho(\zeta))\leq e^{3C_K} e^{-d(o,\tau_\rho(g_n(\zeta)))}.$$

In particular, we have $d_o(\tau_\rho(\xi),\tau_\rho(\zeta))\leq e^{3C_K+K} d_1(\xi,\zeta)^{\frac 1K}$, and $\tau_\rho$ is a $\frac 1K$-H\" older homeomorphism from $\partial \F_r$ to $\Lambda_\rho$. 
\end{lemma}
\begin{proof}
Take $\xi$ and $\zeta$ as in the statement. In the Cayley graph of $\F_r$, for all $k> n$, the triangle 
$(e, g_k(\xi),g_k(\zeta))$ is a tripod
with $g_n(\zeta)=g_n(\xi)$ the common point to the three sides. Denote by $p_n:= \tau_\rho(g_n(\xi)) = \tau_\rho(g_n\zeta))$ this point.

Apply the orbit map $\tau_\rho$. By \Cref{stability_geodesic}, the point $p_n$ is at distance at most $C_K$ of the three geodesic segments $[o,\tau_\rho(g_k(\xi))]$, $[o,\tau_\rho(g_k(\zeta))]$ and $[\tau_\rho(g_k(\zeta)),\tau_\rho(g_k(\xi))]$, see \Cref{fig:tripod}. It implies:
\begin{equation}\label{3.4}
    |\langle \tau_\rho(g_k(\xi)),\tau_\rho(g_k(\zeta))\rangle_o - d(o,p_n)|\leq 3C_K.
\end{equation}

\begin{figure}[t]
    \centering
    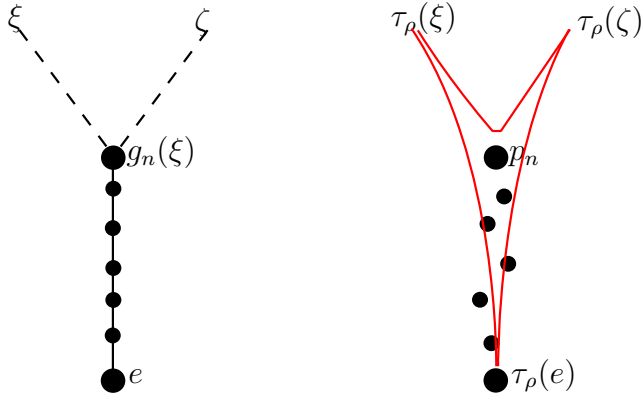
    \caption{A tripod in $\F_r$ (left); its image in $X$ by $\tau_\rho$ (right). In the right hand side picture, the distance between $p_n$ and each of the geodesic is at most $C_K$.}
    \label{fig:tripod}
\end{figure}

Indeed, we have for each of these geodesics in $X$:
\begin{align*}
    |d(o, \tau_\rho(g_k(\xi)))-(d(o,p_n)+d(p_n,\rho(g_k(\xi))o))|&\leq 2C_K\\
    |d(o, \tau_\rho(g_k(\zeta)))-(d(o,p_n)+d(p_n,\tau_\rho(g_k(\zeta))))|&\leq 2C_K\\
    |d(\tau_\rho(g_k(\xi)),\tau_\rho(g_k(\zeta)))-(d(\tau_\rho(g_k(\xi)),p_n)+d(p_n,\tau_\rho(g_k(\zeta))))|&\leq 2C_K
\end{align*}
and these bounds imply \eqref{3.4}.
Letting $k\to\infty$, we have $\tau_\rho(g_k(\xi))\to \tau_\rho(\xi)$ and $\tau_\rho(g_k(\zeta))\to \tau_\rho(\zeta)$. By continuity of the Gromov product, we get the bounds on $d_o(\tau_\rho(\xi),\tau_\rho(\zeta))= \lim_{k\to\infty} e^{-\langle\tau_\rho(g_k(\xi)),\tau_\rho(g_k(\zeta))\rangle_o}$.

The Hölder property then comes from the $K$-isometry property for the orbit map:
$$d(o,\rho(g_n(\zeta))o)\geq \frac 1K d_\Gamma(e,g_n(\zeta))-K=\frac nK-K.$$
\end{proof}
In particular, the image by $\tau_\rho$ of a cylinder $C(\sigma_1\cdots\sigma_n)$ has a diameter at most $e^{3C_K+K} e^{-\frac nK}$.

\subsubsection{Gibbs measures for the shift}\label{sec:frho}

The measure on $\Lambda _ \rho$ as in Lemma \ref{hausdorffmeasure} we are looking for will be defined as $(\tau _{\rho})_*  m_\rho$ where $m_\rho$ is the Gibbs measure on $\partial \F_r$ associated to an appropriate H\"older function 
$f_\rho : \partial \F_r \to \R$ depending on the Schottky representation $\rho$. 

Gibbs measures on $\partial \F_r$ are shift-invariant measures, for the shift $\mathcal S$ defined above, and are characterized by the measures of cylinders. To state it, we need a notation for Birkhoff sums. For a function $f:\partial \F_r\to \R$ and an integer $n$, we denote by $S_n(f)(\zeta) :\partial \F_r\to\R$ the function defined, for $\zeta\in \partial\F_r$, by:
$$S_n(f):=f(\zeta) + f(\mathcal S \zeta) + .....+f (\mathcal S ^{n-1} \zeta).$$
The theory of Gibbs states (see \cite{Bowen-IHES} and references therein) gives\footnote{Here, the actual constants are not pertinent for our work: the notation $m_f (C(\sigma_1\cdots\sigma_n)) \asymp e^{\delta S_{n} f (\zeta)}$ means that there is a constant $c>1$ such that:
$$
\frac 1c e^{\delta S_{n} f (\zeta)}\leq m_f (C(\sigma_1\cdots\sigma_n)) \leq c e^{\delta S_{n} f (\zeta)}.
$$}:
\begin{theorem}\label{gibbs}
Given $f : \partial \F_r \to \R$ a H\"older function, there exists a unique finite $\mathcal S$-invariant measure $m_f$ on $\partial \F_r$ and a unique $\delta >0$ such that for every $n\geq 1$, every cylinder $C(\sigma_1\cdots\sigma_n)$ and every 
$\zeta \in C(\sigma_1\cdots\sigma_n)$, we have
\begin{equation}\label{gibbsmeasure}
m_f (C(\sigma_1\cdots\sigma_n)) \asymp e^{\delta S_{n} f (\zeta)}.
\end{equation}
Moreover, for all $\epsilon >0$, there exists $\alpha$ such that, for every H\"older functions $f_1, f_2$ with $\sup _{\zeta \in \partial \F_r} |f_1 (\zeta) - f_2 (\zeta) | \leq \alpha$, the associated exponents $\delta _1, \delta _2 \in \mathbb R _+$ satisfying (\ref{gibbsmeasure}) verify: $$|\delta _1 - \delta _2| \leq \epsilon.$$
\end{theorem}

We still need to choose the H\"older functions to work with. Given the Schottky representation $\rho$ of 
$\F_r$ in $\Isom\, X$, 
we consider $f_\rho : \partial \F_r \to \R$ for $\zeta = (\sigma _1, \ldots) \in \partial \F_r$ by
\begin{equation}\label{frho}
f_\rho (\zeta) := B(\rho (\sigma _1)o,o,\tau_\rho (\zeta)).
\end{equation}
Notice that $f_\rho$ is H\"older by Lemma \ref{homeoH} and 
the Lipschitz property of the Busemann function w.r.t. its last variable. 
We denote by $\delta_\rho$ the associated exponent given by \Cref{gibbs}. We have a formula for the Birkhoff sums of $f_\rho$, using Busemann functions:
\begin{lemma}\label{birkhoff-busemann}
Let $\zeta = (\sigma_1,\ldots,\sigma_n,\ldots)\in\partial \F_r$. Then we have:
$$S_n f_\rho (\zeta) = B(\rho (g_n(\zeta)) o, o, \zeta).$$
\end{lemma}
\begin{proof}
By definition of $f_\rho$, we have
\begin{equation}\label{birkhoffcocycle}
S_n f_\rho (\zeta) = \sum_{k=1}^n B(\rho (\sigma_k) o, o, \tau_\rho (\mathcal{S} ^{k-1}\zeta)).
\end{equation}
Notice that by construction we have the equalities
$\tau_\rho (\mathcal S^k \zeta) = (\rho (g_k(\zeta)) )^{-1} \tau_\rho (\zeta)$ and $g_{k-1}(\zeta)\sigma_k=g_k(\zeta)$. The invariance property (\ref{busemaninvariance}) applied to each terms of \eqref{birkhoffcocycle} leads to
\begin{equation}\label{birkhoffcocycle2}
S_n f_\rho (\zeta) = \sum_{k=1}^n B(\rho (g_k(\zeta)) o, \rho (g_{k-1}(\zeta))o, \tau_\rho (\zeta)).
\end{equation}
Using repeatedly the cocycle relation $B(x,y,\zeta) =B(x,z,\zeta) +B(z,y,\zeta)$ for the Busemann function, \eqref{birkhoffcocycle2} simplifies into the statement of the Lemma.
\end{proof}

We affirm, following Bowen, that $(\tau_\rho)_* m_{f_\rho}$ gives measure roughly $r^{\delta_\rho}$ to balls of radius $r$ in $\Lambda_\rho$, thus allowing to apply \Cref{hausdorffmeasure}. This is not quite clear at this point: Eq.  \eqref{gibbsmeasure} gives an estimate for the images of cylinders, but we need to compare balls in $\Lambda_\rho$ and images of cylinders. This is done in the following section.

\subsection{Balls and cylinders in limit sets}

\subsubsection{Measuring images of cylinders}

An important step for Bowen's strategy is the quasi-conformal character of $\tau_\rho$: the image by $\tau_\rho$ of cylinders is roughly a ball of radius $e^{S_n(f_\rho)}$. Recall that we assume that the orbit map $\tau_\rho$ is a $K$-quasi-isometry, for some $K\geq 1$. Denote by $\mathcal B^{\Lambda_\rho}_o(\tau_\rho \zeta, r)$ the ball of radius $r$ in $\Lambda _\rho$ centered at $\tau_\rho \zeta$.
\begin{proposition}\label{diam}
For every $n\geq 1$, every cylinder $C(\sigma_1\cdots\sigma_n)$ and every $\zeta \in C(\sigma_1\cdots\sigma_n)$, we have
$$
\mathcal B_o^{\Lambda_\rho} (\tau_\rho\zeta, e^{-7C_K} e^{S_{n} f_\rho (\zeta)})\subset \tau_\rho (C(\sigma _1\cdots \sigma _n)) \subset \mathcal B_o^{\Lambda_\rho} (\tau_\rho\zeta, e^{7C_K}e^{S_{n} f_\rho (\zeta)}),
$$
\end{proposition}

\begin{proof}
Let $\zeta$ be a point in the cylinder $C(\sigma_1\cdots\sigma_n)$.
Denote by $[o, \tau_\rho \zeta)$ the geodesic
ray joining $o$ and $\tau_\rho \zeta$. Using \Cref{stability_geodesic}, the points $\rho (g_k(\zeta)) \cdot o$ (for $k\geq 1$) remain at distance at most $C_K$ of the geodesic ray $[o, \tau_\rho \zeta)$.
In particular, denoting by $p_n$ a point on $[o,\tau_\rho(\zeta))$ at distance $\leq C_K$ from $\rho(g_n(\zeta))\cdot o$, we have:
\begin{align*}
B(p_n,o,\tau_\rho\zeta) &= -d(p_n,o)\\
d(\rho(g_n(\zeta))\cdot o,o)-C_K&\leq d(p_n,o)\leq d(\rho(g_n(\zeta))\cdot o,o)+C_K\\
|B(g_n(\zeta)\cdot o,o,\tau_\rho \zeta) &- B(p_n,o,\tau_\rho \zeta)| \leq C_K.
\end{align*}
As $B(g_n(\zeta)\cdot o,o,\tau_\rho\zeta) = S_n f_\rho\zeta$ by \Cref{birkhoff-busemann}, this leads to: 
\begin{equation}\label{estimate_birkhoff}
    |S_n f_\rho\zeta + d(\rho(g_n(\zeta))\cdot o,o)| \leq 2C_K.
\end{equation}

Using \Cref{homeoH} and the previous Eq. \eqref{estimate_birkhoff}, one get that
$$\tau_\rho (C(\sigma _1,\ldots, \sigma _n)) \subset B_o^{\Lambda_\rho} (\tau_\rho\zeta, e^{3C_K}e^{-d(\rho(g_n(\zeta))\cdot o,o)})\subset B_o^{\Lambda_\rho} (\tau_\rho\zeta,e^{5C_K}e^{S_{n} f_\rho (\zeta)}).$$

Conversely, from the last point of \Cref{stability_geodesic}, we know that, for all $k\leq n$, $\rho(g_k(\zeta))\cdot o$ is at distance at most $C_K$ of the geodesic segment $[o,\rho(g_n(\zeta))\cdot o)]$ and so:
\begin{equation}\label{eq:lowerbounddistance}
\textrm{For }k\leq n,\, d(\rho(g_k(\zeta))\cdot o,o)\leq d(\rho(g_n(\zeta))\cdot o,o)+C_K.
\end{equation}
Suppose that $\xi\in\partial \F_r$ verifies $d_o(\tau_\rho\zeta,\tau_\rho\xi)\leq e^{-7C_K}e^{S_n(f_\rho)}$ and consider the integer $l:=-\ln(d_1(\xi,\zeta))$ for which $g_l(\xi)=g_l(\zeta)$. Using the lower bound in \Cref{homeoH}, we get:
$$e^{-3C_K}e^{-d(o,\rho(g_l(\zeta))\cdot o)}\leq e^{-7C_K}e^{S_n(f_\rho)}.$$
Using also \eqref{estimate_birkhoff} and taking the log, it leads to the bound:
$$d(o,\rho(g_l(\zeta))\cdot o)\geq d(o,\rho(g_n(\zeta))\cdot o)+2C_K.$$
From \eqref{eq:lowerbounddistance}, we deduce that $l\geq n$ and $\xi$ belongs to $C(\sigma_1\cdots\sigma_n)$. We have proven:
$$
\mathcal B_o^{\Lambda_\rho} (\tau_\rho\zeta, e^{-7C_K} e^{S_{n} f_\rho (\zeta)})\subset \tau_\rho (C(\sigma _1\cdots \sigma _n)).$$
\end{proof}

These estimates yields that images of cylinders approximate well enough balls in $\Lambda_\rho$. It will 
allow to compare their measure to measure of balls.

\subsubsection{Measure of balls and an estimate of Hausdorff dimension}

We now prove that $\Hdim \Lambda _\rho = \delta$ using Lemma \ref{hausdorffmeasure}.
To this end, we want to see the Gibbs measure $m_{f_\rho}$ as a measure on $\Lambda _\rho$ and we define
\begin{equation}\label{gibbslambda}
m := (\tau_\rho )_{*} m_{f_\rho}.
\end{equation}
From (\ref{gibbsmeasure}) in Theorem \ref{bowen} and \Cref{diam}, the images $\tau_\rho C(\sigma _1\cdots\sigma _n)$ of cylinders 
 satisfy (with the implied constants only depending on $K$):
\begin{equation*}
{\rm diam}_o \tau_\rho C(\sigma _1,...,\sigma _n) \asymp e^{S_n f_\rho (\zeta)}
\end{equation*}
and
\begin{equation}\label{measurefinal}
m (\tau_\rho C(\sigma_1....\sigma_n)) \asymp e^{\delta_\rho S_{n} f_\rho (\zeta)}.
\end{equation}
In order to conclude that $\Hdim \Lambda _\rho = \delta_\rho$ using Lemma \ref{hausdorffmeasure}, we need to prove estimates as above for balls in $\Lambda _\rho$ in place of $\tau_\rho C(\sigma_1,\ldots,\sigma_n)$. This is given by the following consequence of the previous \Cref{diam}. 
\begin{lemma}\label{ballestimate}
There exists a constant $C'>0$ such that,
for every $0<r<e^{-2C_K-K}$, for every $\zeta\in \partial \F_r$ we have
$$
C'^{-1} r^\delta \leq m (\mathcal B^{\Lambda_\rho}_o(\tau_\rho \zeta, r)) \leq C' r^\delta,
$$
where $m$ is the measure defined in (\ref{gibbslambda}).
\end{lemma}

\begin{proof}
From \Cref{estimate_birkhoff} and the fact that $\tau_\rho$ is a $K$-quasi isometry, we can derive that, for any $\zeta\in\partial \F_r$, we have, for all $n$: $$S_n(f_\rho)(\zeta)\leq - \frac nK +K\textrm{ and }$$
\begin{align*}
    |S_{n+1}(f_\rho)(\zeta) - S_n(f_\rho)(\zeta)|&\leq 4C_K + d(\rho(g_{n+1}(\zeta))\cdot o,\rho(g_n(\zeta))\cdot o))\\&\leq 4C_K+2K
\end{align*}
Fix $0<r<e^{-2C_K-K}$. Then, there exists a largest $n\geq 0$ such that 
$$r\leq e^{-10C_K+S_n(f_\rho)(\zeta)}.$$
We have, from the second point above: 
$$0\leq-10C_K+S_n(f_\rho)(\zeta)-\log(r) \leq 4C_K+2K.$$
Let $N> n$ be the first integer such that 
$e^{10C_K+S_N(f_\rho)(\zeta)}\leq r.$
Once again, we have $$\log(r)-10C_K-S_N(f_\rho)(\zeta) \leq 4C_K+2K$$
Now, from \Cref{diam}, if $\zeta=(\sigma_1,\ldots,\sigma_n\ldots,\sigma_N,\ldots)$, we have:
$$C(\sigma_1,\ldots,\sigma_N)\subset \mathcal B^{\Lambda_\rho}_o(\tau_\rho \zeta, r) \subset C(\sigma_1,\ldots,\sigma_n).$$
It follows that, denoting by $c$ the implicit constant in \eqref{measurefinal}, the measure of $B^{\Lambda_\rho}_o(\tau_\rho \zeta, r)$ verifies:
\begin{align*}
1/c e^{-\delta_\rho(14C_K +2K)} r^{\delta_\rho}  \leq 1/c e^{\delta_\rho S_N(f_\rho)(\zeta)}&\leq m(B^{\Lambda_\rho}_o(\tau_\rho \zeta, r))\\
m(B^{\Lambda_\rho}_o(\tau_\rho \zeta, r))
&\leq c e^{\delta_\rho S_n(f_\rho)(\zeta)}\leq c e^{\delta_\rho(14C_K +2K)} r^{\delta_\rho}.
\end{align*}
This proves the Lemma.
\end{proof}
From the previous \Cref{ballestimate}, the measure  $m= (\tau_\rho)_\star m_{f_\rho}$ on $\Lambda_\rho$ verifies exactly the hypothesis of \Cref{hausdorffmeasure}, for the parameter $\delta_\rho$ associated to the function $f_\rho$ by \Cref{gibbs}. To summarize, we have obtained, following Bowen's strategy, the Hausdorff dimension of the limit set:
\begin{proposition}\label{prop:delta_rho}
With the notations above, we have:
$$\Hdim \Lambda_\rho = \delta_\rho.$$
\end{proposition}
The last point to prove \Cref{thm:bowen-CAT(-1)} is to understand how $f_\rho$ changes when $\rho$ varies along a sequence verifying the assumption of \Cref{thm:bowen-CAT(-1)}.

\subsection{Convergence of Hausdorff dimensions}

We now focus on the proof of \Cref{thm:bowen-CAT(-1)}. The two previous subsections focused on a single representation $\rho$. We now consider throughout this subsection a sequence $R_k : \F_r\to \Isom\, X$ (for $k\in\N$) and an additional representation $R_\infty : \F_r\to \Isom\, X$. To each of these representations is associated its orbit map $\tau_{R_k}:\partial \F_r \to \partial X$ (see \Cref{coding}) and the Hölder function $f_{R_k}:\partial \F_r\to \R$ (see \Cref{sec:frho}), that we will abbreviate to $\tau_k$  and $f_k$, for $k\in\N\cup\{\infty\}$. We assume that they verify:
\begin{itemize}
    \item There is a constant $K$ such that the orbit map $\tau_{\infty} : \F_r \to X$ is a $K$-quasi isometric embedding.
    \item $\tau_k\xrightarrow{c.s.}\tau_\infty$, i.e. for all $g\in \F_r$, we have:
    $$\tau_k(g) = R_k(g) \cdot o \xrightarrow{k \to \infty} \tau_\infty(g) = R_\infty(g) \cdot o.$$   
\end{itemize}

\subsubsection{Hausdorff convergence of limit sets}

Recall that,for $k\in \N\cup \{\infty\}$ and $\xi=(\sigma_1,\ldots) \in \partial \F_r$, we have defined (see Eq. \eqref{frho}):
$$f_k(\xi) = B(R_k(\sigma_1)o,o,\tau_{k}(\xi)).
$$
In view of the second part of \Cref{gibbs}, the following proposition is crucial:
\begin{proposition}\label{prop:convergenceLimitSets}
With the notations above, for any $\epsilon>0$ there is an integer $L>0$ such that for all $k\geq L$, all $\xi \in \partial \F_r$, we have
$$d_o(\tau_{k}(\xi),\tau_{\infty}(\xi))\leq \epsilon.
$$
In particular, the limit sets $\Lambda_{R_k}$ converge Hausdorff to $\Lambda_{R_\infty}$ and
the maps $f_k$ converge uniformly on $\F_r$ to $f_\infty$.
\end{proposition}
This proposition holds because, in a $\CAT(-1)$-space, one can approximate the points $\tau_{_k}(\xi)$ that appear in the definition of $f_k$ using only a finite part of the orbit of $o$ under $R_k(\F_r)$.
Indeed, \Cref{approxlimitset2} describes in a quantitative way how the limit set $\Lambda _\Gamma \subset \partial X$ of any discrete group $\Gamma$ acting on a $\CAT(-1)$ space $X$ can be approximated by rays from the origin $o$ to points in finite subsets of its orbit $\Gamma \cdot o \subset X$. The proof of the proposition uses the same ideas.

\begin{proof}[Proof of  \Cref{prop:convergenceLimitSets}]
Fix $0<\epsilon<1$. 

By \Cref{thm:quasi_isometric_open}, there exists a constant such that all but finitely many $R_k$ are quasi-isometric representations. For simplicity and w.l.o.g., we assume that all $R_k$ are $K$-quasi-isometric representations.
In particular, there exists $n$ big enough such that for all $g\in\F_r$ of length $\geq n$, and all $k\in \N\cap \{\infty\}$, we have $d(o,R_k(g)\cdot o)$ big enough to verify: 
\begin{equation*}
    e^{-\frac{d(o,R_k(g)\cdot o)}{2}}\leq \epsilon
\end{equation*}

Using the fact that the orbit maps $\mathcal \tau_{k}$ converge pointwise, we can choose $k_0$ big enough so that for all $g\in\F_r$ of length $\leq n$, all $k\geq k_0$, we have $d(R_k(g)\cdot o, R_\infty(g)\cdot o)\leq C_K$.

Now fix $\xi=(\sigma_1,\ldots)\in \partial \F_r$. Consider $g_n:=g_{n}(\xi) = \sigma_1\cdots \sigma_n$.
For all $k\in\N\cup\infty$, from \Cref{stability_geodesic}, the geodesic ray $[o,\tau_k(\xi))$ passes at distance $\leq C_K$ from $R_k(g_n)\cdot o$.

By the triangular inequality, we obtain that $[o,\tau_k(\xi))$ passes at distance $\leq C_K+d(R_k(g_n)\cdot o,R_\infty(g_n)\cdot o)\leq 2C_K$ from $R_\infty(g_n)\cdot o$. With the terminology of shadows (see \Cref{sec:visualmetric}), both $\tau_k(\xi)$ and $\tau_\infty(\xi)$ belong to the shadow:
$$\pi_o^{\partial X}(\mathcal B(\rho_\infty(g_n)\cdot o,2C_K)).$$
From \Cref{prop:diam_shadows}, with $r = 2C_K$, we obtain that there is a constant $c_{2C_K}$ such that:
$$d_o(\tau_{k}(\xi),\tau_{\infty}(\xi))\leq c_{2C_K}\epsilon.
$$
This proves the first point. The claim about Hausdorff convergence of limit sets follows immediately.

Moreover, for any $\xi =(\sigma_1,\ldots) \in\partial \F_r)$, we have:
$$|f_k(\xi)-f_\infty(\xi)| = |B(R_k(\sigma_1)o,o,\tau_k(\xi))-B(R_\infty(\sigma_1)o,o,\tau_\infty(\xi))|.$$
Now the Busemann function behaves nicely w.r.t each variables (see \Cref{coding}); $\sigma_1$ can only take a finite numbers of values, so that $R_k(\sigma_1)o$ converges uniformly (on $\partial \F_r$) to $R_\infty(\sigma_1)o$; and, by the first estimate of this proposition, $\tau_k(\xi)$ converges uniformly to $\tau_\infty(\xi)$. This proves the uniform convergence of $f_k$ to $f_\infty$.
\end{proof}
We have now done all the work toward the proof of \Cref{thm:bowen-CAT(-1)} and its corollary. 

\subsubsection{Proof of \Cref{thm:bowen-CAT(-1)} and its corollary}

We now pull the strings together to finish the proofs of \Cref{thm:bowen-CAT(-1)} and its corollary. Let us begin with the main theorem:
\begin{proof}[Proof of \Cref{thm:bowen-CAT(-1)}]
Take a sequence of representations $R_k$ (for $k\in\N$) and an additional representation $R_\infty : \F_r\to \Isom\, X$ that verify:
\begin{itemize}
    \item There is a constant $K$ such that the orbit map $\tau_\infty : \F_r \to X$ is a $K$-quasi isometric embedding, for $k\in \N\cup \{\infty\}$.
    \item The sequence of orbit maps $(\tau_k)_{k\in \N}$ converges point-wise to $\tau_\infty$.   
\end{itemize}
Then, by \Cref{prop:convergenceLimitSets}, the limit sets $\Lambda_k$ converge for the Hausdorff topology to $\Lambda_\infty$. This proves the first point of \Cref{thm:bowen-CAT(-1)}.

Second, we know that:
\begin{itemize}
    \item The Hausdorff dimension of each $\Lambda_k$ verify $\Hdim \Lambda_k=\delta_k$, where $\delta_k$ is the coefficient associated to $f_k$, see \Cref{prop:delta_rho}.
    \item The functions $f_k$ converge uniformly to $f_\infty$ on $\partial \F_r$, see the last point of \Cref{prop:convergenceLimitSets}.
\end{itemize}
The last point of \Cref{gibbs} gives the convergence $\delta_k\to\delta_\infty$, which gives the conclusion of the Theorem:
$$\Hdim \Lambda_k \to \Hdim \Lambda_\infty.$$
\end{proof}

In order to get \Cref{coro:BowenAlgConv}, we need to check that the two assumptions (recalled in the previous proof) of \Cref{thm:bowen-CAT(-1)} follow from convergence in $\RepX(\F_r)$ toward a limiting representation in $\SchottkyspaceX(\F_r)$.
\begin{proof}[Proof of \Cref{coro:BowenAlgConv}]
Take a sequence of representations $R_k$ (for $k\in\N$) that converges in $\RepX(\F_r)$ to a Schottky representation $R_\infty : \F_r\to \Isom\, X$. First, $R_\infty$ is assumed Schottky, so by definition its orbit map is a quasi-isometric embedding.

Then, convergence in $\SchottkyspaceX(\F_r)$ amounts to point-wise convergence $R_k(\gamma)\to\rho_\infty(\gamma)$. This last convergence, in $\Isom\, X$, is itself equivalent to the point-wise convergence: for all $x\in X$, $R_k(\gamma)x\to R_\infty(\gamma)x$ in $X$.
Taking $x=o$, we get the point-wise convergence of the orbit maps $\tau_{k}$ to $\tau_{\infty}$: So the second assumption of the \Cref{thm:bowen-CAT(-1)} is also fulfilled.

By \Cref{thm:bowen-CAT(-1)}, we indeed get $\Hdim \Lambda_k \to \Hdim \Lambda_\infty.$
\end{proof}

This ends our extension of Bowen theorem. We now turn to the other step of our strategy: turning diverging sequences of representations to convergent ones. We begin by a review of the needed material.

\section{Kernels, embeddings and asymptotic cones}\label{sec:constructionRk}

\subsection{Kernels and embeddings in hyperbolic spaces}\label{sec:EmbeddingKernel}

We review in this section the work \cite{Monod-Py-2019} of Monod and Py on the self-representations of the infinite dimensional Möbius group, based on the notion of \emph{kernel of hyperbolic type} introduced by Gromov.

\subsubsection{Kernels and embeddings}

The more classical notion of kernel of \emph{positive type} inspired the following definition:
\begin{Definition}[\cite{Monod-Py-2019} Definition 3.1]
Given a set $X$, a kernel of hyperbolic type is a function $\beta : X\times X \rightarrow \mathbb R$ which is symmetric, non negative, taking the value $1$ on the diagonal and satisfying
\begin{equation*}
\sum_{i,j =1}^n  c_i c_j \beta (x_i, x_j) \leq \left(\sum_{i,j =1} ^n c_k \beta (x_0, x_k)\right)^2
\end{equation*}
for all $n \in \N$, all $x_0, x_1, ..... , x_n \in X$ and all $c_1, ..., c_n \in \mathbb R$.
\end{Definition}
A remarkable feature of any such kernel is that it is associated to an embedding of $X$ in a hyperbolic space $\H^\alpha$ (see \Cref{H-infini}). We only state the case where $X$ is a separable topological space, for which $\alpha$ is at most countable:
\begin{theorem}[{\cite[Theorem 3.4]{Monod-Py-2019}}]\label{thm:MonodPy}
Let $X$ be a separable topological space with a continuous kernel $\beta$ of hyperbolic type. Then there exists an unique cardinal $\alpha$, at most countable, and a continuous embedding $f_\beta:X\to \H^\alpha$, unique up to $\Isom\,\H^\alpha$, such that $f_\beta$ has total hyperbolic image and for all $x,y \in X$:
\begin{equation*}
\beta (x,y) = \cosh {d(f_\beta(x), f_\beta(y))}
\end{equation*}
\end{theorem}

\begin{remark}\label{total-n}
The embedding given in the above Theorem \ref{thm:MonodPy} has the following "dimension property": given distinct points $x_0, x_1, ...,x_n$ in $X$, the subspace of $\H^\alpha$ generated by the points $f_\beta (x_0), ... ,f_\beta (x_n)$ has dimension $n$, cf. \cite[Theorem 13.1.1 (iv)]{DSU} in the case $X$ is a tree, the general case being analogous.
\end{remark}
As the hyperbolic space
$\H ^\alpha$ and the map $f_\beta$ are uniquely determined up to isometry, there is a representation
of $\Aut (X, \beta)$ in $\Isom\,\H ^\alpha$ and the map $f_\beta : X \rightarrow \H^\alpha$ is equivariant with respect to that representation. 

A fundamental property of hyperbolic kernels is the following theorem:
\begin{theorem}[\cite{Monod-Py-2019} Theorem 3.10]\label{power-t}
Let $\beta$ be a kernel of hyperbolic type on a set $X$. Then so is $\beta ^t$ for all $0\leq t \leq 1$.
\end{theorem}

\subsubsection{Examples for hyperbolic spaces and trees}\label{sec:kernels-examples}

As a first example of kernel of hyperbolic type, one may consider $X=\H^N$ and $\beta (x,y) := \cosh {d(x,y)}$. We therefore 
obtain by Theorems \ref{thm:MonodPy} and \ref{power-t}  a one parameter family of continuous embeddings $f_t : \H^N \rightarrow \H^\infty$, for all $0< t < 1$ associated to the kernels 
$\beta ^t (x,y) := \left(\cosh {d(x,y)}\right)^t$,
where $\H ^\infty$ is the (unique up to isometry) separable infinite dimensional hyperbolic space. For $0< t < 1$, we denote by $\nu_t:\Isom\, \H^N\to \Isom\, \H^\infty$ the representation for which the map $f_t$ is equivariant. Note that the maps $f_t$ are continuous. Moreover, they are very close to being simple dilatations and are anyway $1$-Lipschitz, as stated in the following:

\begin{proposition}\label{QI1}
Let $f_t : \H^N \rightarrow \H^\infty$, $0< t < 1$ be the maps associated to the kernels 
$\beta ^t (x,y) := \left(\cosh {d(x,y)}\right)^t$. Then, $f_t$ is a $(t,\ln 2)$-quasi-isometry and we have for all $x,y \in \H^N$ and all $0< t \leq 1$:
\begin{equation*}
t \, d(x,y) \leq \, d(f_t(x), f_t (y)) \leq \min [ t\, d(x,y) +\ln 2 \, \,, d(x,y)]
\end{equation*}
\end{proposition}

\begin{proof}
We observe for all $x,y \in \H ^N$ and all $0< t \leq 1$,
\begin{equation*}
\frac{e^{d(f_t(x),f_t (y))}}{2} \leq \cosh {d(f_t(x), f_t(y))} =  \left(\cosh {d(x,y)}\right)^t \leq e^{t\,d(x,y)},
\end{equation*}
hence 
$$
d(f_t(x),f_t (y) \leq t \,d(x,y) + \ln 2.
$$
On the other hand, since $0<t\leq 1$, $$ \cosh {d(f_t(x), f_t(y))} =  \left(\cosh {d(x,y)}\right)^t \leq \cosh {d(x,y)}$$
hence we also have 
$$
d(f_t(x), f_t(y)) \leq d(x,y).
$$
Similarly, setting $s := \frac{1}{t} \geq 1$, we have by convexity of $u \rightarrow u^s$
\begin{equation*}
\cosh {d(x,y)} = (\cosh {d(f_t(x), f_t(y))^{s}} \leq \frac{1}{2} \left(e^{sd(f_t(x), f_t(y))} + e^{-sd(f_t(x), f_t(y))}\right)
\end{equation*}
hence 
$$
\cosh {d(x,y)} \leq \cosh \left({s \,d(f_t(x), f_t(y)}\right)
$$
therefore
$$
t\, d(x,y) ) \leq d(f_t(x),f_t (y)).
$$
\end{proof}

\medskip

Another important family of examples of kernel of hyperbolic type arises for $(X,d)$ a metric tree. We consider here only the case of a separable metric tree.
The kernel is defined as $\beta _s (x,y) := e^{s d(x,y)}$ with $s\geq 0$, \cite[Proposition 1.5]{Monod}.
As before we have a map $F_s : X \rightarrow \H ^\infty$, equivariant for a representation $\mu_s:\Isom\, T \to \Isom\, \H^\infty$ and with total hyperbolic image in $\H ^\infty$, such that
\begin{equation}\label{distance3}
e^{s d(x,y)} = \cosh {d(F_s(x), F_s (y))}
\end{equation}
for all $x,y \in X$, where we have denoted by $d$ both distances on $X$ and on $\H ^\infty$.

As in the previous case, we have 
\begin{proposition}\label{QI2}
Let $F_s : X \rightarrow \H ^\infty$ be the maps associated to the kernels $\beta _s (x,y) := e^{s d(x,y)}$ on the metric tree $(X,d)$.
Then $F_s$ is a $(s, \ln 2)$-quasi-isometry. More precisely, for all $x,y \in X$, all $s > 0$, 
\begin{equation}\label{Fs1}
s \, d(x,y) \leq d(F_s (x), F_s (y)) \leq s \, d(x,y) + \ln 2.
\end{equation}
Moreover, if $d(x,y) \leq [0, \frac{\ln 2}{s}]$, we have
\begin{equation}\label{Fs2}
d(F_s (x), F_s (y)) \leq \left(2s \, d(x,y)\right)^{1/2} .
\end{equation}
\end{proposition}
\begin{proof}
The proof of (\ref{Fs1}) is analogous to the one of Proposition \ref{QI1}, we leave it to the reader. By the relation (\ref{distance3}), we see 
that $\cosh{d(F_s(x), F_s (y))} \in [1,2]$ if and only if $d(x,y) \in [0, \frac{\ln 2}{s}]$. When these conditions are satisfied, we have
\begin{align*}
s d(x,y) = \ln \left(\cosh{d(F_s(x), F_s(y))}\right) &\geq \cosh{d(F_s(x), F_s(y))} -1 \\ &\geq \frac{d(F_s(x), F_s(y)) ^2 }{2}
\end{align*}
which implies (\ref{Fs2}).
\end{proof}

The two previously defined maps $f_t : \H ^N \to \H^\infty$ and $F_s : X \to \H ^\infty$ extend to equivariant embeddings of $\partial f_t:\partial \H^N\to \partial \H ^\infty$ and $\partial F_s: \partial T\to \partial \H ^\infty$, as all spaces are $\CAT(-1)$. Notice that these embeddings are bilipschitz homeomorphisms on their images, \cite{Monod-Py-2019}.

\subsubsection{Close kernels and related embeddings}

In the sequel, we will need to compare the maps induced by different kernels of hyperbolic type on different metric spaces and show some continuous dependence of the maps $f$
with respect to their associated kernels.  The notion of close maps taking values in a metric space is straightforward:
\begin{definition}
    Let $K$ be a set, $(X,d)$ a metric space and $\epsilon>0$. Two maps $f,g:K\to X$ are $\epsilon$-close if $\forall k \in K$ we have $d(f(k),g(k))\leq \epsilon$. 
\end{definition}
The definition of close kernels requires a bit more work. Let us consider two metrics spaces $(X_1, d_1)$ and  $(X_2,d_2)$ and assume that $X_1$ is locally compact. Let $\beta _i$ be kernels of hyperbolic type on $X_i$, for $i=1,2$. Fix a compact subset $K\subset X_1$ and an embedding $\phi : K \hookrightarrow X_2$. Our notion of close kernels is that they take approximately the same values on points in $K$ as for points in $\phi(K)$:
\begin{definition}
For any $\eta>0$, we say that the kernels $(\beta_1,\beta_2)$ are $(K,\phi,\eta)$-close if for any $x,y\in K$, we have: 
\begin{equation}\label{hypothesisapproxkernel}
    | \beta _2 (\phi (x), \phi (y)) - \beta _1 (x,y) | \leq \eta.
\end{equation}
\end{definition}
We claim that close hyperbolic kernels imply close embeddings in $\H^\infty$, up to an isometry of $\H^\infty$. Let 
$f_i : X_i \rightarrow \H ^{\alpha _i}$ be the embeddings associated to the kernels $\beta_i$, for $i=1,2$. This proposition is not surprising but will prove crucial later on.
\begin{proposition}\label{approxkernel}
Let $(X_1,d_1)$  be a separable metric space and fix $\epsilon>0$ and $K\subset X_1$ a finite subset.
Then there exists $\eta:=\eta(\epsilon,K)>0$ such that for any metric space $(X_2,d_2)$, any embedding $\phi : K \hookrightarrow X_2$, any pair of kernels $(\beta_1,\beta_2)$ that are $(K,\phi,\eta)$-close and associated maps $f_i : X_i \rightarrow \H ^{\infty}$, there exists an isometry
\[\mathcal F: \H^{\infty} \to \H^{\infty}\]
such that the maps $\mathcal F\circ f_1$ and $f_2\circ \phi$ from $K$ to $\H^{\infty}$ are $\epsilon$-close.
\end{proposition}
The proof of this proposition relies on the following Lemma, which is a variation of classical results.
Recall that the hyperbolic space $\H ^\infty$ can be seen as the 
set of positive lines in $\mathbb R \oplus H$ for the bilinear 
form $B \left( (s,x),(s', x') \right) := s s' -\langle x, x' \rangle$ where 
$(H , \langle ., . \rangle)$ is a (separable) Hilbert space. Given a set of points $\{u_0, \ldots, u_n\}$ in $\H ^\infty$, we define 
$H(u_0, \ldots, u_n)$ to be the smallest hyperbolic subspace $\H ^m \subset \H ^\infty$ containing $\{u_0, \ldots,u_n\}$.
\begin{lemma}\label{approxhyperbolic}
Let $\{u_0, \ldots u_n\}\subset \H ^\infty$ verify $\dim(H(u_0, \ldots,u_n)=n)$ and fix $\epsilon >0$.
Then there exists $\eta:= \eta(\epsilon, u_0,\ldots, u_n) >0$ such that for every subset $\{v_0, \ldots, v_n\}$ in $\H^\infty$ with $ \dim H (v_0, \ldots, v_n) =n$ and, for all $0\leq i,j\leq n$,
\begin{equation}\label{approxhyperbolicassumption}
| \cosh d(u_i, u_j) - \cosh d (v_i, v_j) | \leq \eta ,
\end{equation}
there exists an isometry $\mathcal F : \H^\infty \rightarrow \H^\infty$ sending $H(u_0, ...,u_n)$ onto $H(v_0,...,v_n)$ such that $\mathcal F (u_0) = v_0$, 
and for all $i = 1, ..., n$, 
$$
d(\mathcal F (u_i), v_i) \leq \epsilon.
$$
\end{lemma}
The proof relies on classical constructions in Hilbert spaces.
\begin{proof}
Let us write $u_i : =(s_i, x_i)$ and $v_i : =(t_i, y_i)$ in  $\H^\infty$, $i =0,...,n$, with:
$$B (u_i, u_i) = s_i ^2 - \langle x_i, x_i \rangle = 1 \quad \textrm{ and } \quad 
B (v_i, v_i) = t_i ^2 - \langle y_i, y_i \rangle = 1.$$
We can assume $u_0 = v_0 = (1,0) \in \H^\infty$.
Recall that for $w,w' \in \H^\infty$, the distance between $w$ and $w'$ satisfies:
\begin{equation}\label{cosh=B}
\cosh d(w,w') = B (w,w').
\end{equation}
We fix $\epsilon >0$ and assume 
(\ref{approxhyperbolicassumption}) with $\eta= \eta (\epsilon, u_0,...,u_n)$ to be determined later. From (\ref{approxhyperbolicassumption}) applied to 
$u_0, u_i$ and $v_0, v_i$, we get for $i=1,...,n$  
\begin{equation}\label{t-s}
    |t_i -s_i | \leq \eta
\end{equation}
and 
\begin{equation}\label{scalarproduct-approx}
    |\langle y_i, y_j\rangle  - 
     \langle x_i, x_j\rangle  |\leq \eta_1 (\eta)
\end{equation}
with $\lim _{\eta \to 0} \eta _1 (\eta) =0.$

We consider $V:= {\rm Vect} (y_1,...,y_n)$ and $\overline V := {\rm Vect}(x_1,...,x_n)$ the two $n$-dimensional subspaces of $H$ generated by  
$\{y_1,...,y_n\}$ and
$\{x_1,...,x_n \}$. Let 
$\{e_1,...,e_n\}$ and $\{\overline{e}_1,...,\overline{e}_n \}$ the Gram-Schmidt orthonormal basis of 
$\{y_1,...,y_n\}$ and
$\{x_1,...,x_n \}$.

According to the Gram-Schmidt algorithm, there exist universal rational functions $C_{i,k} \in \mathbb Q(X_1,...,X_{n^2})$ such that 
$y_i = \sum_{k=1}^n c_{i,k} e_k$ and 
$x_i = \sum_{k=1}^n \overline{c}_{i,k} \overline{e}_k$, with 
$c_{i,k}:= C_{i,k}((\langle y_l, y_{l'} \rangle)_{l,l'})$ and
$\overline{c}_{i,k}:= C_{i,k} ((\langle x_l, x_{l'} \rangle)_{l,l'}) $. 
Therefore, we get from (\ref{scalarproduct-approx}) that 
\begin{equation}\label{coeff-gramschmidt}
\left(\sum_{i,k} |c_{i,k} - \overline{c}_{i,k}|^2\right)^{1/2} \leq \eta_2 (\eta),
\end{equation}
with $\lim_{\eta \to 0} \eta _2(\eta) =0$.

Let us now consider the isometry 
$\varphi : \overline{V} \rightarrow V$ defined by $\varphi (\overline{e}_i) = e_i$, $i=1,...,n$. We can extend $\varphi$ to an isometry of the Hilbert space $H$, still denoted by $\varphi$. By (\ref{coeff-gramschmidt}), $\varphi$ satisfies 
\begin{equation}\label{norm-beta}
\|\varphi (x_i) - y_i\| \leq \eta _2 (\eta).
\end{equation}
The map $\mathcal F : \H^\infty \rightarrow \H^\infty$ defined by $\mathcal F (t,y) := (t, \varphi (y))$ is an isometry  such that 
$\mathcal F (u_i) = \mathcal F ((s_i,x_i)) = (s_i, \varphi (x_i))$. From (\ref{cosh=B}), we have
$\cosh d(\mathcal F (u_i), v_i)=\cosh d((s_i,\varphi (x_i)), (t_i,y_i))=B (\mathcal F (u_i), v_i)$, therefore, by (\ref{t-s}) and (\ref{norm-beta}),
$$
|\cosh d(\mathcal F (u_i), v_i) -1| =|t_i s_i - \langle \varphi (x_i), y_i\rangle -1| \leq \eta _3 (\eta),
$$
where $\lim_{\eta \to 0} \eta _3 (\eta) =0$. To conclude the proof of \Cref{approxhyperbolic}, we can choose $\eta = \eta(\epsilon, u_0,\ldots,u_n)$ such that $\eta_3(\eta)\leq \epsilon$.
\end{proof}
\begin{proof}[Proof of Proposition \ref{approxkernel}]
Let $K:= \{ x_0, x_1, ...,x_n\} $ and fix $\epsilon >0$. To prove Proposition \ref{approxkernel} we want to apply Lemma \ref{approxhyperbolic}
to $u_l := f_1(x_l)$ and $v_l := f_2 \circ \phi (x_l)$. In particular, we take $\eta:=\eta(\epsilon,K)>0$ to be the $\eta(\epsilon,u_0,\ldots,u_n)$ given by Lemma \ref{approxhyperbolic}. We first notice that, by  \Cref{total-n},
$$\dim H(f_1(x_0), \cdots ,f_1 (x_n))=
\dim H( f_2 \circ \phi (x_0), \cdots,f_2 \circ \phi (x_n)) =n.$$

Let us check the assumption \eqref{approxhyperbolicassumption}. 
Consider the maps $f_1 ; X_1 \to \H ^{\infty}$ and $f_2 : X_2 \to \H ^{\infty}$ associated to the kernels
$\beta _1$ and $\beta _2$. Recall that the maps $f_1, f_2$ satisfy for every $x,y \in X_i$, ($i=1,2$):
$$
\beta _i ( x,y) = \cosh d(f_i (x), f_i (y)).
$$
Therefore, by hypothesis (\ref{hypothesisapproxkernel}), we have for every $0\leq l,l'\leq n$:
\begin{align*}
|\beta_2(v_{l'}, v_l) &- \beta_1(u_{l'}, u_l)| =
\\
&|\cosh d(f_2 \circ \phi (x_{l'}), f_2 \circ  \phi (x_l)) - \cosh d (f_1 (x_{l'}), f_1 (x_{l}))| \leq \eta.
\end{align*}
Lemma \ref{approxhyperbolic} implies the existence of an isometry 
$\mathcal F : \H ^{\infty} \rightarrow \H ^{\infty}$ sending $H(u_0, ...,u_n)$ onto $H(v_0,...,v_n)$ and such that $\mathcal F (u_0) = v_0$, 
and for all $i = 1, ..., n$, 
$$
d(\mathcal F (u_i), v_i) \leq \epsilon,
$$
which is the isometry needed to prove Proposition \ref{approxkernel}.
\end{proof}
This concludes our review of kernels of hyperbolic type and associated embeddings.

\subsection{Asymptotic cones of diverging representations}\label{sec:asymptotic-cones}

We now consider a diverging sequence of discrete representations $\rho _i : \Gamma \rightarrow \Isom\,\H^N$ of a finitely generated group $\Gamma$. Informally,
'diverging' implies that the sequence of representations converges to an action of $\Gamma$ on a metric tree $(X,d)$, cf. \cite{Paulin-ENS}, \cite{Bestvina}. Our aim in this section is to show that this
convergence naturally extends to kernels of hyperbolic type and their associated maps into $\H ^\alpha$.

    \subsubsection{Asymptotic cones}
    
    Let us introduce the set up and the notations we need, following \cite{Paulin-ENS}. We fix a non principal ultrafilter $\omega$ on the set of integers $\N$
and for any bounded sequence $(\lambda _{k})_{k \in \N}$ of real numbers, denote $\omegalim \lambda _k$ the limit of this sequence along the ultrafilter $\omega$.
Let us consider $(X_k, d_k, o_k)$ a sequence of pointed metric spaces and denote
$$
X_\infty := \{ (x_k)_{k\in \N} \,|\, x_k \in X_k, \,\, \omegalim d_k(x_k, o_k) < \infty \}.
$$
This space inherits a pseudo-metric $d_\infty$ defined by $$d_\infty((x_k),(y_k)) = \omegalim d_k(x_k,y_k).$$ This function is symmetric, verifies the triangular inequality, but may take the value $0$ for two different sequences. A set with a pseudo-metric always gives rise to a metric space by "quotienting by the metric". Indeed, one defines the ultralimit 
$$(X_\omega , d_\omega, o_\omega) := \omegalim (X_k, d_k, o_k)$$ of the sequence $(X_k, d_k, o_k)$ as the quotient of the pointed space 
$(X_\infty, d_\infty, (o_k))$ by the equivalence relation $(x_k) \sim (y_k)$ iff $d_\infty((x_k), (y_k)) =0$.
Note that the distance $d_\omega$ between two classes $[(x_k)]$ and $[(y_k)]$ is well-defined as the distance $d_\infty$ between any representative sequences $(x_k)$ and $(y_k)$.

Given a sequence of  isometric actions of a group $\Gamma$ on $(X_k, d_k)$ such that for all $\gamma \in \Gamma$,
\begin{equation}\label{group-action}
\omegalim d_k (\gamma o_k, o_k) <\infty ,
\end{equation}
one gets an isometric action of $\Gamma$ on $(X_\omega, d_\omega)$ by setting
$\gamma \cdot [(x_k)] := [(\gamma (x_k))]$. In that case, we will say that the sequence of actions of $\Gamma$ on $(X_k, d_k)$ converges to an action of $\Gamma$ 
on $(X_\omega, d_\omega)$.

This general construction may be applied to a variety of cases. For our concerns, the main example is the case of asymptotic cones, where $X_k$ is assumed constant and the metrics $d_k$ are just a rescaling. In fact, we will restrict here to the case where $X_k=\H^N$ ($N$ a positive integer). 
\begin{definition}
    Let $N$ be a positive integer and $(\lambda _k)$ be a sequence of positive real numbers converging to $0$. The associated asymptotic cone of $\H^N$ is defined as:
\begin{equation*}
\Cone_\omega (\H ^N, \lambda _k \, d, o_k) := \omegalim (\H ^N, \lambda _k \, d, o_k).
\end{equation*}
\end{definition} 

It is known that for every sequence $\lambda _k$ converging to $0$, the asymptotic cone $$\Cone_\omega (\H ^N, \lambda_{k} d, o_k)$$ is a metric tree \cite{Paulin-ENS} and is not separable if $N>1$.

\subsubsection{Limits of diverging representations}\label{diverging}
    
    A celebrated application of the ultralimit construction is the study of diverging sequences of representations. Throughout this section, we fix a non-principal ultrafilter $\omega$ on $\N$. We work with a fixed finite dimensional hyperbolic spaces $(\H^N,d)$ and a finitely generated group $\Gamma$ with symmetric generating set $S$.
\begin{definition}
Let $\rho _k : \Gamma \rightarrow \Isom\,\H^N$ be a sequence of representations. The \emph{minimal joint displacement} of $\rho_k$ is:
\begin{equation}\label{joint-displacement}
r_k := \inf _{x\in \H^N} \max _{s \in S} d(x, \rho _k (s) x).
\end{equation}
The sequence is \emph{diverging} if the minimal joint displacement tends to $+\infty$.
\end{definition}
Let us outline how such a sequence of representations converges to an action on a real tree, for details we refer for example to \cite{Bestvina}, \cite{Paulin-Top-Ap}.
The idea is to consider the asymptotic cone for the sequence of rescalings  $\lambda_k = r_k ^{-1}$. We can choose a sequence of points $o_k \in \H^N$ almost realizing the minimal joint displacement: 
\begin{equation*}
 \max _{1\leq j \leq \nu} d(o_k , \rho _k (s_j) o_k) \leq r_k +1.
\end{equation*}
By our choice of the rescaling factors $r_k^{-1}$, the isometric actions of $\Gamma$
on $(\H ^N, r_{k}^{-1} d, o_k)$ satisfy the condition (\ref{group-action}). Therefore $\Gamma$ acts isometrically on the asymptotic cone
$\Cone_\omega (\H ^N, r_{k}^{-1} d, o_k)$.
To sum up, the sequence of actions of $\rho _k (\Gamma)$ on $(\H ^N, r_{k}^{-1} d, o_k)$  converges to an action of $\rho _\infty (\Gamma)$ on the metric tree $\Cone_\omega (\H ^N, r_{k}^{-1} d, o_k)$.

Moreover there exists a unique separable minimal closed and convex subtree $$(T, d_\omega, o_\omega) \subset \Cone_\omega (\H ^N, r_{k}^{-1} d, o_k)$$
which is invariant under the action of $\rho _\infty (\Gamma)$. This minimal subtree is the union of the axes in $\Cone_\omega (\H ^N, r_{k}^{-1} d, o_k)$
of the hyperbolic elements of $\rho _\infty (\Gamma)$, see \cite[Proposition 2.4]{Paulin-Top-Ap} or \cite{DK}. 

We will also need to understand the relation between the convergence in the sense of ultrafilter as above with the equivariant Gromov-Hausdorff convergence
introduced by F. Paulin, \cite{Paulin-ENS}.
Let us consider a sequence of metric spaces $(X_k, d_k)$ and a metric space $(X,d)$ with an isometric action of a group $\Gamma$. 
\begin{definition}[\cite{Paulin-ENS}]\label{paulin}
The sequence $(X_k, d_k)$  converges to $(X,d)$ in the \emph{equivariant Gromov-Hausdorff topology} if the following property is satisfied:
for every $\epsilon >0$, every finite subset $K\subset X$, every finite subset $P\subset \Gamma$, there exists a subset 
$\Omega \subset \N$ belonging to the ultrafilter $\omega$ such that for every $k\in \Omega$, there exists a finite subset $K_k \subset X_k$,
a relation $\mathcal R_k  \subset K_k \times K$, whose projection on each factor is surjective, such that for every $g, h \in P$, $x_k \mathcal R _k x$ and $y_k \mathcal R _k y$, 
we have 
$$
|d_k (g x_k, h y_k) - d (gx,hy) | \leq \epsilon.
$$
\end{definition}

Paulin showed that a sequence of metric spaces converges for this topology to its ultralimit.
\begin{theorem}[\cite{Paulin-ENS} Proposition 2.1 (4)]\label{conv-ultra-gromov-haus-equiv}
Let $(X_k, d_k, o_k)$, be pointed metric spaces with an isometric action of a group $\Gamma$ and $(X,d,o)$ be their ultralimit. Assume that 
the condition (\ref{group-action}) is satisfied, then 
the sequence $(X_k, d_k, o_k)$ converges to $(X,d, o)$ in the equivariant Gromov-Hausdorff topology.
\end{theorem}
Notice that if $(X',d, o) \subset (X,d, o)$ is a closed convex $\Gamma$-invariant subset, then $(X_k, d_k, o_k)$ also converges 
to $(X', d, o)$ in the equivariant Gromov-Hausdorff topology. In other words, this topology is not separated. In particular, the sequence also converges to the minimal subtree:
\begin{corollary}\label{GH-convergencerep}
Let  $\rho _k : \Gamma \rightarrow \Isom\,\H^N$ be a sequence of representations with minimal joint displacement $r_k$ satisfying $\lim_{k\to \infty} r_k = \infty$ and let
$(T, d_\omega, o_\omega)$ be the minimal $\rho _\infty (\Gamma)$-invariant separable subtree of  $\Cone_\omega (\H ^N, r_{k}^{-1} d, o_k)$. 

Then, the sequence $(\H ^N, r_{k}^{-1} d, o_k)$ converges to $(T, d_\omega, o_\omega)$ in the equivariant Gromov-Hausdorff topology.
\end{corollary}

\subsubsection{Convergence of kernels and asymptotic cones}
    
    We consider here a sequence of rescaled hyperbolic spaces
$(\H^N, r_{k}^{-1} d, o_k)$ with $r_k$ tending to infinity.
Recall that its asymptotic cone $\Cone_\omega (\H ^N, r_{k}^{-1} d, o_k)$ is a tree and contains a minimal invariant separable subtree, that we denote $(X,d_\omega)$. As such, for every $s > 0$, the function $$\beta _s (x,y) = e^{sd_\omega(x,y)}$$ is a kernel of hyperbolic type, see Section \ref{sec:kernels-examples}.
There exists an equivariant map $$F_s : X \rightarrow \H ^{\infty}$$
satisfying
$$
e^{sd(x,y)} = \cosh{d(F_s (x), F_s (y))}.
$$
On the other hand, for every $0< t < 1$, the function $$\beta ^t (x,y) = \cosh {d(x,y)}^t$$
is a kernel of hyperbolic type on $\H ^N$ with associated equivariant map 
$$
f_t : \H^N \rightarrow \H ^\infty ,
$$
satisfying
$$
\cosh {d(x,y)}^t = \cosh{d(f_t (x), f_t (y))}.
$$

The following proposition states that a suitable choices of a sequence $t_k$ makes those kernels converge.
\begin{proposition}\label{prop-convergence}
Fix $s>0$ and let $r_k$ be a sequence of real numbers tending to infinity. Set $t_k := s r_{k}^{-1}$. Then, 
for every $x=[(x_k)]$ and $y=[(y_k)]$ in $\Cone _\omega (\H ^N, r_k ^{-1}d, o_0)$, we have: 
$$\omegalim \beta^{t_k} (x_k, y_k) = \beta _s (x, y).$$
\end{proposition}
\begin{proof}
Let $x=[(x_k)]$ and $y=[(y_k)]$ be two points in $\Cone _\omega (\H ^N, r_k ^{-1}d, o_0)$. 

\medskip

\noindent
$\bullet$ Case 1. Assume that $x=y$, i.e. $d_\omega ([(x_k)], [(y_k)]) = \omegalim r_k ^{-1} d(x_k, y_k) =0$. Notice that 
$$
1\leq \cosh{d(x_k, y_k)} ^{t_k} \leq e^{t_{k} d(x_k, y_k)},
$$
therefore, since $t_k = s r_k ^{-1}$, we get
$$
\omegalim \cosh{d(x_k, y_k)} ^{t_k}  =1,
$$
hence 
$$
\omegalim \cosh{d(x_k, y_k)} ^{t_k}  =  e^{s d_\omega (x,y)},
$$
that is,
$$
\omegalim \beta^{t_k} (x_k, y_k)   = \beta _s  (x, y).
$$

\noindent
$\bullet$ Case 2. We assume now that $d_\omega ([(x_k)], [(y_k)]) = \omegalim r_k ^{-1} d(x_k, y_k) >0$. In that case, we have
$$
\omegalim e^{-d(x_k, y_k)} =0,
$$
therefore, as $t_k\to 0$,
$$
\omegalim \cosh{d(x_k, y_k)} ^{t_k} =  \omegalim \left( \frac{e^{d(x_k, y_k)}}{2}\right) ^{t_k} =  e^{s d_\omega (\bar x, \bar y)},
$$
thus,
$$
\omegalim \beta^{t_k} (x_k, y_k)  =  \beta _s  (x, y).
$$
This finishes the proof of the proposition.
\end{proof}

We will prove in \Cref{sec:partial-hausdorff-convergence} that for every $s\geq 0$, the maps 
$f_{t_k} : X_k \rightarrow \H ^\infty$ converge,
in a suitable sense, to the map $F_s : T \rightarrow \H ^{\alpha}$ for $t_k := s r_{k}^{-1}$. But before we provide a version of \Cref{thm:bowen-CAT(-1)} with ultrafilters.

\subsubsection{A version with ultrafilters of Bowen Theorem}\label{sec:Bowen-ultrafilter}

Let us fix an ultrafilter $\omega$. We claim that, in Theorem \ref{intro-bowen}, if one assumes that the orbital maps $(\tau _k)_{k\in\N}$ only converges to $\tau _\infty$ along the ultrafilter, then the conclusions remain unchanged upon replacing limits by $\omega$-limits.  
The proof of this new version is identical to the one of Theorem \ref{thm:bowen-CAT(-1)}, replacing limits by $\omega$-limits. Indeed, the crucial point is \Cref{prop:convergenceLimitSets}; in this proposition, once fixed $\epsilon$, only a finite number of points in the orbit $R_k(\F_r)\cdot o$ matter. As an ultrafilter is stable under finite intersections, if $\tau_k\xrightarrow{c.s.}\tau_\infty$ along the ultrafilter, one get that $\omegalim \max_{\partial\F_r}|f_k-f_\infty|=0$. This leads to
\begin{theorem}[Bowen theorem for $\omega$-limits]\label{thm:Bowen-omega}
Let $(X,d,o)$ be a complete pointed $\CAT(-1)$-space; consider a sequence $(R_k)_{k\in \N}$ in $\RepX(\F_r)$ and $R_\infty\in \SchottkyspaceX(\F_r)$ a Schottky representation.
Assume that the sequence of orbit maps $(\tau_k)_{k\in \N}$ converges point-wise to $\tau_\infty$ along the ultrafilter $\omega$: 
    $$\forall g\in\F_r,\quad \omegalim \tau_k(g)= \tau_\infty(g).$$   
Then we have:
\begin{enumerate}
    \item The sequence of limit sets $(\Lambda_k)_{k\in \N}$ of $(R_k)_{k\in \N}$ converges along $\omega$, in the Hausdorff topology of $\partial X$, to the limit set $\Lambda_\infty$ of $R_\infty$.
    \item $\omegalim \Hdim \Lambda_k= \Hdim \Lambda_\infty$.
\end{enumerate}
\end{theorem}

\section{Partial Hausdorff convergence of embeddings}\label{sec:partial-hausdorff-convergence}

We consider a finitely generated group $\Gamma$ and a diverging sequence of representations $\rho_k \in \Repr_{\H^N} (\Gamma)$ with minimal joint displacement $r_k$. We have recalled the construction of the asymptotic cone, and the limiting action of $\Gamma$ on a separable minimal tree $T$, see \Cref{sec:asymptotic-cones}. We have recalled that the sequence of pointed metric spaces $(X_k,d_k,o_k):=(\H^N,r_k^{-1}d,o_k)$ converges in the equivariant Gromov-Hausdorff topology to $(T,d,o)$, see \Cref{paulin}.
We have also defined, for every $s\geq 0$, the embeddings  
$f_{t_k} : X_k \rightarrow \H ^\infty$ and $F_s : T \rightarrow \H ^{\alpha}$ for $t_k := s r_{k}^{-1}$, which are equivariant for representations $\nu_{t_k}$ and $\mu_s$, see \Cref{sec:EmbeddingKernel}. 

The goal of this section is to introduce a notion of convergence of a sequence embeddings, the \emph{partial equivariant Hausdorff convergence} such that the images of $f_{t_k}$ converge to the image of $F_s$. Recall from \Cref{nohausdorffconv} that it may not be possible that the Hausdorff distance between the images goes to $0$; so our notion will be weaker than plain Hausdorff convergence.
Let us consider more generally a sequence of 
pointed separable metric spaces 
$(X_k, d_k, o_k)$ for $k\in \N$
and $(X,d,o)$ a pointed separable metric space, with isometric actions 
$\rho_k : \Gamma \rightarrow \Isom\, X_k$ and 
$\rho : \Gamma \rightarrow \Isom\, X$, such that $X_k$ converges to $X$ in the equivariant Gromov-Hausdorff topology, cf. \Cref{paulin}.
The idea is to choose a sequence of embeddings $\phi_k$ and sequences of ever bigger finite sets $K_k\subset X_k$, 
such that the Hausdorff limit of the sequence $\phi_k (K_k)$ is the whole image $\phi (X)$ of the space $X$, this convergence being equivariant. Before setting the definition, we introduce some notations.

Let us fix an ultrafilter $\omega$. Let us consider a metric space $\mathcal H$ and (non necessarily isometric) embeddings $\phi :X \hookrightarrow \mathcal H$ and 
$\phi_k : X_k \hookrightarrow \mathcal H$. We consider isometric actions $\rho _\infty \in \RepX (\Gamma)$, $\rho _k \in \Repr_{X_k} (\Gamma)$ and $R_k , \, R_\infty \in \Repr_{\mathcal H} (\Gamma)$ such that the embeddings $\phi$ and $\phi_k$ are equivariant: for every $\gamma \in \Gamma$, 
$\phi \circ \rho _\infty (\gamma)= 
R_\infty (\gamma) \circ \phi$, and 
$\phi _k \circ \rho_k (\gamma) = R_k (\gamma) \circ \phi_k$.

\begin{definition}[Partial equivariant Hausdorff convergence]\label{def:partial-Hausdorff-convergence}
Given a metric space $\mathcal H$ and an equivariant embedding 
$\phi : X \hookrightarrow \mathcal H$, we say that a sequence of equivariant embeddings $\phi_k : X_k \hookrightarrow \mathcal H \, \, , k\in \N$, realizes a partial equivariant Hausdorff convergence of $\phi_k (X_k)$ to $\phi (X)$ in $\mathcal H$ along the ultrafilter if there exists a sequence of finite subsets $K_k\subset X_k$, $k\in\N$, such that:
\begin{itemize}
    \item Hausdorff convergence along $\omega$: $\phi(X)$ is the closure of the set of $\omega$-limit of sequences $(x_k)$, where $x_k\in K_k$:
            $$
            \phi (X) = \overline{\{\omegalim \phi_k(x_k),\; \textrm{ for sequences }(x_k)\textrm{ with }x_k\in K_k\}},$$
    \item Equivariance along $\omega$:  for every $\gamma \in \Gamma$, $x \in X$ and sequence $x_k \in X_k$ such that $\phi(x) = \omegalim_k \phi _k (x_k) $, we have 
    $$\phi (\rho(\gamma)\cdot x) = \omegalim _k \phi _k (\rho_k(\gamma) \cdot x_k).$$
\end{itemize}
\end{definition}

\subsection{Partial equivariant Hausdorff convergence of renormalizations}

The following Theorem states that, up to choosing isometries of $\H^\infty$, the sequence $f_{t_k}$ realizes a partial equivariant Hausdorff convergence of $f_{t_k}(X_k)$ to $F_s(T)$. 
\begin{theorem}[Partial Hausdorff convergence]\label{thm:partial-hausdorff-convergence}
For any diverging sequence $\rho_k \in \Repr_{\H^N} (\Gamma)$, with the notations above, there exist, for $k\in \N$, isometries $\mathcal I_k \in \Isom\,\H^\infty$, representations $R_k$ and $R_\infty$ in $\Repr _{\H^\infty}(\Gamma)$, such that the sequence $\phi _k := \mathcal I _k \circ f_{t_k}$ realizes a partial equivariant Hausdorff convergence of $\phi _k(X_k)$ to $F_s(T)$ in $\H^\infty$ along the ultrafilter. 
\end{theorem}

The sections \ref{sec:lifts} and \ref{sec:approx} are devoted to the proof of this Theorem. Our approach relies heavily on the separability of the minimal tree $T$.
We fix, throughout this section, an increasing sequence of finite subsets of $T$:
$$E_0=\{o\}\subset E_1 \subset ..... \subset E_l \subset E_{l+1} \subset....$$ such that $\# E_l = l+1$, and 
$E := \cup_l E_l$ is a countable dense subset of $T$. In the sequel, we will denote its points by
$$
E_l = \{x^i \,|\, 0\leq i \leq l \};
$$
this notation makes sense as $E_l\subset E_{l+1}$. We will also fix 
once and for all an increasing sequence $(\Gamma_k)_{k\in \N}$ of finite subsets of $\Gamma$ such that $\Gamma_0=\{e\}$ and $\cup_\N \Gamma_k=\Gamma$. Moreover, for a subset $S_k\subset X_k$, we define the subset
\[\Gamma_k\cdot S_k:=\{\rho_k(\gamma)(s)|s\in S_k, \gamma\in \Gamma_k\}.\]
As $\Gamma$ acts on $T$ through $\rho_\infty$, we denote by $\Gamma_l\cdot E_l$ the set of all images of points in $E_l$ by elements of $\Gamma_l$.

We will use the equivariant Gromov-Hausdorff convergence for each pair $(E_l,\Gamma_l)$ as well as \Cref{approxkernel}; we will try to make $l$ go to $\infty$ as well as $k$ for achieving the partial equivariant Hausdorff convergence. This can be viewed as an application of the diagonal extraction procedure. 

\subsection{A sequence of lifts of ever bigger sets}\label{sec:lifts}

For $\epsilon>0$ and $l\in \N$, denote by $\eta (\epsilon, \Gamma_l\cdot E_l)$ the number given by Proposition \ref{approxkernel}. 
By Theorem \ref{conv-ultra-gromov-haus-equiv} and Proposition \ref{prop-convergence}, there exists a subset $I(l, \epsilon)$ in $\omega$ verifying the property of the definition of equivariant Gromov-Hausdorff convergence. The actual definition\footnote{This definition is easier to understand assuming that the relations $\mathcal R_{k,l,\epsilon}$ are maps $E_l\to X_k$: it would say that one can lift, for any $k\in I(l,\epsilon)$, the set $E_l$ in a set $K_{k,l,\epsilon}$ such that the relative distances, resp. values of kernels, between points in $\Gamma_l\cdot E_l$ and their images in $\Gamma_l \cdot K_{k,l,\epsilon}$ coincide up to $\epsilon$, resp $\eta(\epsilon,E_l)$.} of $I(l,\epsilon)$ reads:
\begin{definition}\label{def:K_ile}
For an integer $l$ and $\epsilon>0$, we define the set $I(l, \epsilon)$ as the set of integers $k$ such that there is a subset $K_{k,l,\epsilon}\subset X_k$, a relation $\mathcal R_{k,l,\epsilon} \subset E_l\times K_{k,l,\epsilon}$ that verifies, for every $x^i, x^{i'} \in E_l$, every $\gamma^i,\gamma^{i'}\in \Gamma_l$ and every $x^i_{k,l,\epsilon}$, $x^{i'}_{k,l,\epsilon}$ in $K_{k,l,\epsilon}$ related to respectively $x^i$ and $x^{i'}$,  the two inequalities:
\begin{align}
|d_T (\gamma^{i}\cdot x^i,\gamma^{i'}\cdot x^{i'}) - d_{X_k} (\gamma^{i}\cdot x^{i}_{k,l,\epsilon}, \gamma^{i'}\cdot x^{i'}_{k,l,\epsilon})| \leq & \epsilon.\nonumber\\
|\beta_s (\gamma^{i}\cdot x^i, \gamma^{i'}\cdot x^{i'}))- \beta ^{t_k}(\gamma^{i}\cdot x^{i}_{k,l,\epsilon}, \gamma^{i'}\cdot x^{i'}_{k,l,\epsilon})| \leq & {\eta (\epsilon, \Gamma_l\cdot E_l)} ,\label{approx-hilbert}
\end{align}
\end{definition}
Note that for all $\epsilon$, we have $I(0,\epsilon)=\N$ as one can always isometrically lift a point in $T$ to any $X_k$, equivariantly under the action of $\Gamma_0=\{e\}$. Moreover, by Theorem \ref{conv-ultra-gromov-haus-equiv} and Proposition \ref{prop-convergence}, for all $\epsilon>0$ and $l\in \N$, the set $I(l,\epsilon)$ belongs to the ultrafilter $\omega$.

We first state the following easy monotony property of the family of sets $I(l,\epsilon)$. 
\begin{lemma}\label{mono_Il}
If $l'\geq l$ and $\epsilon'\leq \epsilon$, we have $I(l',\epsilon')\subset I(l,\epsilon)$.
\end{lemma}
Indeed, if you know how to lift a lot of points, with a strict condition $\epsilon$, you can do it for an easier situation.

For the rest of our construction, we further choose once and for all a decreasing sequence $(\epsilon_l)_{l\geq 1}$ of positive real numbers converging to $0$. We will only use the sequence $I(l,\epsilon_l)$ and we fix the choice of the lifts $K_{k,l}:=K_{k,l\epsilon(l)}\subset X_k$ as in Definition \ref{def:K_ile}. As the relations $\mathcal R_{k,l,\epsilon_l}$ project surjectively on each factor, we fix functions $E_l\to K_{k,l}$, for $k\in I(l,\epsilon_l)$, denoted $x^i\mapsto x^i_{k,l}$, such that we have $(x^i,x^i_{k,l})\in \mathcal R_{k,l,\epsilon_l}$. We furthermore assume, up to applying an isometry of $X_k$, that for each $k$ and $l$, we have $x^0_{k,l}$ is the origin $o_k$.

We define the function $\ell : \N \rightarrow \N\cup\{\infty\}$ by:
\begin{equation*}
\ell (k) := \max \{ l \in \N \, | \, k \in I(l, \epsilon _l)\}.
\end{equation*}
Note that $\ell(k)$ is in fact finite:
\begin{lemma}\label{lem:no-tree-in-Xi}
For every $k\in \N$, we have $\ell (k) <\infty$.
\end{lemma}
\begin{proof}
Let us assume, by contradiction that $\ell (k) = \infty$ for some $k\in \N$, ie. $k\in I(l, \epsilon _l)$ for every $l\in \N$. By the local compacity of $X_k$ (recall that $x_{k,l}^0=o_k$), we can assume that there is a sequence of integers $l_j$ tending to $\infty$ such that for every $i\in \N$, $\lim_{j \to \infty} x^i_{k, l_j} = x^i_{k, \infty} \in X_k$. The points $x^i_{k, \infty} \in X_k$ satisfy, for every $x^i, x^{i'} \in E$, 
$$
d_k(x^i_{k, \infty}, x^{i'}_{k, \infty}) = d_T(x^i, x^{i'}). 
$$
The subset $\overline{\{x^i_{k, \infty} \, , i\in \N\}} \subset X_k$ is therefore isometric to $T$, which is impossible since $X_k$ does not contain isometrically embedded non trivial tripods.
\end{proof}
We prove now that $\ell(k)$ goes to $\infty$ along the ultrafilter.
\begin{lemma}\label{l(i)-infini}
For every $l \in \N$, we have $\{k \in I \,|\, \ell (k) \geq l \} = I(l, \epsilon _l)$. 
In particular,  we have $\omega-\lim _{k \to \infty} \ell (k) = \infty .$
\end{lemma}
\begin{proof}
Fix $l\in \N$.
By definition, we have $I(l, \epsilon _l) \subset \{k \in \N \,|\, \ell (k) \geq l \}$. For the reverse inclusion, from \Cref{mono_Il}, for any $l'\geq 0$, we have the inclusion $I(l+l', \epsilon _{l+l'}) \subset I(l, \epsilon _l)$ since $E_l$ is increasing and $\epsilon _l$ is decreasing.
The assumption $\ell (k) = l+l' \geq l$ means that $k\in I(l+l', \epsilon _{l+l'})$. By the previous remark, it implies that $k\in I(l, \epsilon _l)$.

The second part of the Lemma follows since $I(l, \epsilon _l) \in \omega$: the liminf, along $\omega$, of $\ell(k)$ is bigger than $l$ for any $l$.
\end{proof}

The previous \Cref{{lem:no-tree-in-Xi}} allows to take the biggest meaningful among the $K_{k,l}$'s in $X_k$, for any $k$. Indeed, we define the sequence of compact sets $K_k\subset X_k$ needed for Theorem \ref{thm:partial-hausdorff-convergence}:
\begin{equation*}
K_{k} := K_{k,\ell(k)} = \{x_{k,\ell(k)}^l,\, 0\leq l\leq \ell(k)\} \subset X_k.
\end{equation*}
The $K_k$'s are finite subsets of $X_k$ such that the hyperbolic kernel $\beta ^{t_k}$ between points in $\Gamma_{\ell(k)}\cdot K_k$
coincides up to $\eta (\epsilon _{\ell (k)}, \Gamma_{\ell(k)}\cdot E_{\ell (k)})$ with the hyperbolic kernel $\beta _s$ between related points of $\Gamma_{\ell(k)}\cdot E_{\ell (k)} \subset T$. 
We now construct the sequence of isometries $\mathcal I_k$.

\subsection{A sequence of approximations}\label{sec:approx}

There is an inclusion $E_{\ell (k)} \subset  E_{\ell (k')}$ when $\ell (k) \leq \ell (k')$. However {\it it does not imply the inclusion} $f_{t_k} (K_k) \subset f_{t_{k'}} (K_{k'}) $ since 
the sets $f_{t_k}(K_k)$ are not related as subsets of $\H ^{\infty}$. Indeed, $K_k$ and $K_{k'}$ are subsets of different spaces $X_k$ and $X_{k'}$. The information we have is that the kernel $\beta^{t_k}$ on $K_k$ is close to the kernel $\beta_s$ on $E_{\ell(k)}$ if $\ell(k)$ is big. 
Taking also into account the actions of $\Gamma$, we therefore use Proposition \ref{approxkernel} to send isometrically each $f_{t_k}(\Gamma_{\ell(k)}\cdot K_k)$ close to $F_s(\Gamma_{\ell(k)}\cdot E_{\ell(k)})$.

\begin{lemma}\label{estimate-mathcal-I}
For $k\in \N$, there exists an isometry $\mathcal I_{k} : \H^\infty\to \H^\infty$, verifying:
for every $0\leq l \leq \ell (k)$ and $\gamma\in\Gamma_{\ell(k)}$:
\begin{equation}\label{def:I_n-maps}
d(\mathcal I_{k} \circ f_{t_k} (\rho_k(\gamma)\cdot x_{k,\ell (k)}^l), F_s (\rho_\infty (\gamma)\cdot x^l)) \leq \epsilon _{\ell (k)}.
\end{equation}
In particular, for each $l\in \N$, the sequence $z_k^l := \mathcal I_k\circ f_{t_k} (x_{k,\ell (k)}^l)$ is defined for $k\in I(l,\epsilon_l)$ and its $\omega$-limit is $F_s(x^l)\in \H^\infty$. Moreover, for any $J\subset I(l,\epsilon_l)$ with $\lim_{k\in J} \ell(k)= \infty$, we have the convergence $\lim_{k\in J} z^l_k = F_s(x^l)$.
\end{lemma}
In the previous theorem, when $\ell(k)=0$, the condition on the isometry is almost empty: any isometry $\mathcal I_k$ fixing the origin $o = f_{t_k}(o_k)=f_{t_k}(x^0_{k,\ell(k)})$ does the trick. Note also that a $J$ as in \Cref{estimate-mathcal-I} exists: as $\omegalim_k \ell(k)=+\infty$, there is a subsequence $J\subset I(l,\epsilon_l)$ with $\lim_{k\in J} \ell(k)=+\infty$. But this subset may not belong to the ultrafilter, so the two convergence statements are different. The condition on $J$ is explicit enough to be actually checked in examples, see \Cref{ex:mcmullen3}.

\begin{proof}
Fix $k\in \N$. By construction, Eq. \eqref{approx-hilbert} holds for any $1\leq l, l' \leq \ell (k)$, and $\gamma^l$, $\gamma^{l'}$ in $\Gamma_{\ell(k)}\cdot K_k$. This states that the kernels $\beta_s$ on $\Gamma_{\ell(k)}\cdot E_{\ell(k)}$ and $\beta^{t_k}$ on $\Gamma_{\ell(k)}\cdot K_k$ are $\eta(\epsilon_l,\Gamma_{\ell(k)}\cdot E_{\ell(k)})$-close.
The Proposition \ref{approxkernel} then implies the existence of an isometry
$\mathcal I_{k} : \H ^\infty \to \H ^\infty$ with the desired property.

The last sentence follows from the fact that $z^l_k$ is defined as soon as $\ell(k)\geq l$, or equivalently $k\in I(l,\epsilon_l)$ and the two convergence statement are granted by the first point: the $\omega$-limit holds as we have $I(l,\epsilon_l)$ belongs to the ultrafilter and $\omegalim_k \ell(k)=+\infty$. The second one follows directly from the inequality in the first point. 
\end{proof}

The sequence $\mathcal I_k$ given by the previous Lemma verifies the conditions of Theorem \ref{thm:partial-hausdorff-convergence}:
\begin{proof}[Proof of Theorem \ref{thm:partial-hausdorff-convergence}]
Fix a choice of $\mathcal I_k$  as in Lemma \ref{estimate-mathcal-I} and define $$\phi _k := \mathcal I _k \circ f_{t_k} : \H^N \hookrightarrow \H^\infty.$$ For coherence of notations with \Cref{def:partial-Hausdorff-convergence}, we denote here by $\phi$ the map $\phi:=F_s : T \hookrightarrow \H^\infty.$ 
Define also (recall the representations $\nu_t$ and $\mu_s$ from \Cref{sec:kernels-examples}) the following representations of $\Gamma$ into $\Isom\, \H^\infty$: $$R_k := \mathcal I _k \circ (\nu _{t_k} \circ \rho _k) \circ \mathcal I _k ^{-1} \quad \textrm{ and }\quad R_\infty := \mu _s \circ \rho_\infty.$$ By construction, $\phi_k$ and $\phi$ are equivariant: for $\gamma \in \Gamma$, we have 
$$\phi _k \circ \rho _k (\gamma) = R_k (\gamma) \circ \phi _k\quad \textrm{ and }\quad \phi \circ \rho _\infty (\gamma) = R_\infty (\gamma) \circ \phi.$$

We prove the conditions of  \Cref{def:partial-Hausdorff-convergence} of partial equivariant Hausdorff convergence, for the embedding $\phi$ and the sequence $\phi_k$ using the sequence of sets
$$
K_{k} = \{x_{k,\ell (k)} ^l\, | \, 0\leq l \leq \ell(k)\}.
$$ 
{\bf Hausdorff convergence along $\omega$:} By \Cref{estimate-mathcal-I}, any point $\phi(x^l)$ in $\phi(E)$ is the $\omega$-limit of the sequence $z^l_k:=\phi_k(x_{k,\ell (k)} ^l)$, defined for $k\in I(l,\epsilon_l)$.
Hence, we have
$\phi (E) \subset \overline{\{\omegalim \phi_k(x_k),\; \textrm{ for sequences }(x_k)\textrm{ with }x_k\in K_k\}}$; by continuity of $\phi$, we get a first inclusion:
$$
\phi (T) \subset  \overline{\{\omegalim \phi_k(x_k),\; \textrm{ for sequences }(x_k)\textrm{ with }x_k\in K_k\}}.
$$
We now prove the reverse inclusion. Consider  
$$z\in \overline{\{\omegalim \phi_k(x_k),\; \textrm{ for sequences }(x_k)\textrm{ with }x_k\in K_k\}}.$$
Let us write $z= \omegalim \phi_k \left(x_{k,\ell (k)}^{l_k}\right)$ for a sequence of integers $1\leq l_k \leq \ell (k)$.
We have by Lemma \ref{estimate-mathcal-I}:
$$
d\left(\phi_k \left(x_{k,\ell (k)}^{l_k}\right), 
\phi \left(x^{l_k}\right) \right) \leq \epsilon _{\ell(k)}
$$
where $\phi \left(x^{l_k}\right) \in \phi (E) \subset \phi (T)$ and $\omegalim \epsilon _{\ell(k)} =0$ as $\omegalim \ell(k) = \infty$.
Therefore, we obtain at the limit
$d(z, \phi(T)) =0$. It implies that $z\in \phi (T)$ since $\phi (T)$ is closed as a direct consequence of the Proposition \ref{QI2}. 
This gives the reverse inclusion and ends the proof of the first part of the Definition \ref{def:partial-Hausdorff-convergence}.

\noindent
{\bf Equivariance along $\omega$:} Consider $x\in T$ and a sequence $x_k \in K_k$ such that $\phi (x) = \omegalim_k \phi _k (x_k)$. Fix $\epsilon >0$. From the density of $E$ in $T$ and the continuity of $\phi$, there exists $x^l \in E$  such that $d(\phi (x), \phi (x^l)) \leq \epsilon$. On the other hand, since $\phi (x) = \omegalim_k \phi _k (x_k)$, the set 
$I_\epsilon:= \{k\in \N,\, | \, d(\phi (x), \phi_k (x_k)) \leq \epsilon \}$ belongs to $\omega$. By the triangular inequality and Lemma \ref{estimate-mathcal-I}, we get for every $k\in I_\epsilon\cap I(l,\epsilon_l)$:  
$$d(\phi_k (x_k), \phi _k ( x^l_{k, \ell(k)})) \leq 2\epsilon + \epsilon_{\ell(k)}.$$ 
By equivariance of $\rho_k$, for every $\gamma \in \Gamma$, we have:
\begin{equation*}
d\left(\phi_k \left(\rho_k (\gamma) x_k\right), \phi_k \left(\rho _k (\gamma) x^l_{k,\ell (k)})\right) \right) \leq 2\epsilon + \epsilon_{\ell (k)}.
\end{equation*}
Moreover, if $\gamma \in \Gamma_{\ell(k)}$, \Cref{estimate-mathcal-I} gives:
$$ d\left(\phi_k \left(\rho_k (\gamma)
x^l_{k,\ell(k)}\right), \phi \left(\rho_\infty (\gamma) x^l\right)\right) \leq \epsilon_{\ell(k)}.$$
Therefore, for $k$ verifying: $k\in I_{\epsilon}$, $k\in I(l,\epsilon_l)$ and $\gamma \in  \Gamma_{\ell(k)}$, we have
\begin{align*}
d\left(\phi_k \left(\rho_k (\gamma) x_k\right), \phi \left(\rho _\infty (\gamma)  x\right) \right)  &\leq 
d\left(\phi_k \left(\rho_k (\gamma) x_k\right), \phi_k \left(\rho_k (\gamma)
x^l_{k,\ell(k)}\right)\right) \\
& \phantom{xxxxx} + d\left(\phi_k \left(\rho_k (\gamma)
x^l_{k,\ell(k)}\right), \phi \left(\rho_\infty (\gamma) x^l\right)\right) \\
& \phantom{xxxxx}+d\left(\rho_\infty (\gamma) x^l, \rho_\infty (\gamma) x \right),\\
&\leq  3\epsilon + 2\epsilon_{\ell (k)}.
\end{align*}
The three previous conditions on $k$ are in the ultrafilter, $\epsilon_{\ell(k)}$ goes to zero along the ultrafilter by \Cref{l(i)-infini}, hence the set of $k$ such that $$d\left(\phi_k \left(\rho_k (\gamma) x_k\right), \phi \left(\rho _\infty (\gamma)  x\right) \right)\leq 5\epsilon$$
belongs to the ultrafilter.
This ends the proof of Theorem \ref{thm:partial-hausdorff-convergence}.
\end{proof}

\subsection{Proof of \Cref{hausdimasymp}}\label{sec:hausdimasymp}

At this point, we have proven \Cref{thm:bowen-CAT(-1)} (and its ultrafilter version \Cref{thm:Bowen-omega}), as well as \Cref{thm:partial-hausdorff-convergence}. We conclude the proof of \Cref{hausdimasymp}. Let us restate it now that we defined the notion of $\omega$-limit:
\begin{theorem}\label{hausdimasymp-ultrafiltre}
Let $(\rho _k)_{k\in \N}$ be a diverging sequence in
$\Repr_{\H^N} (\F_{r})$ with minimal joint 
displacement $r_k$. Fix a non-principal ultrafilter $\omega$ on $\N$ and assume that the limiting action of $\rho _\infty (\F _r)$ on the minimal tree $T$ is Schottky. Then 
$$
\omegalim _k \; r_k \cdot\Hdim ( \Lambda _{\rho _k}) = \Hdim ( \Lambda _{\rho _\infty}).
$$
\end{theorem}

\begin{proof}
Consider the embeddings $\phi_k:\H^N\to\H^\infty$ given by \Cref{thm:partial-hausdorff-convergence} applied to the sequence $\rho_k$ (with the choice $s=1$): this sequence realizes an equivariant partial Hausdorff convergence of $\phi_k(\H^N)$ to $F_1(T)$. Recall that the representations $R_k$ and $R_\infty$ from $\F_r$ to $\Isom\,\H^\infty$ coming from the Theorem \ref{thm:partial-hausdorff-convergence} are defined, for all $\gamma\in \F_r$, by: $$R_k(\gamma) = \mathcal I_k \circ (\nu _{t_k} \circ \rho_k(\gamma))\circ\mathcal I_k^{-1} \quad \textrm{ and } \quad 
R_\infty (\gamma) = \mu_1 \circ \rho _\infty (\gamma).
$$
We have to prove the following:
\begin{enumerate}
    \item The orbit maps $\tau_{R_k}$ converge pointwise along the ultrafilter to $\tau_{R_\infty}$;
    \item  $R_\infty$ is Schottky;
    \item for all $k$ we have: $r_k\cdot\Hdim \Lambda_{\rho_k} =  \Hdim \Lambda_{R_k}$ and $\Hdim \Lambda_{\rho_\infty} =  \Hdim \Lambda_{R_\infty}$.
\end{enumerate}
Indeed, the two first points allow to apply the Theorem \ref{thm:Bowen-omega} to the family $(R_k)_{k\in\N}$; it gives that $\omegalim_k \Hdim \Lambda_{R_k} = \Hdim \Lambda_{R_\infty}$.
Then the last point gives the convergence of \Cref{hausdimasymp-ultrafiltre}.

{\bf Proof of point 3:} By definition, $\Hdim \Lambda_{\rho_k}$ is computed using the visual distance for the usual hyperbolic space $\H^N$. But $X_k$ is the renormalization of $\H^N$ by a factor $\frac{1}{r_k}$. So the Hausdorff dimension $\Hdim_{X_k} (\Lambda_{\rho_k})$ for the visual distance of $X_k$ is $\Hdim_{X_k} (\Lambda_{\rho_k})=r_k\Hdim \Lambda_{\rho_k}$. Then, by construction, $\phi_k$ is a $(1,\ln(2))$)-quasi-isometry from $X_k$ to $\H^\infty$ and extends as a Lipschitz map on boundaries. Moreover, by construction, $\phi_k(\Lambda_{\rho_k})=\Lambda_{R_k}$. It implies that $$\Hdim_{X_k} (\Lambda_{\rho_k}) = \Hdim \Lambda_{R_k}.$$
We have proven $r_k\cdot \Hdim \Lambda_{\rho_k} =  \Hdim \Lambda_{R_k}$. The case of $\rho_\infty$ and $R_\infty$ is handled similarly, as $F_1$ gives a Lipschitz homeomorphism between $\Lambda_{\rho_\infty}$ and $\Lambda_{R_\infty}$.

{\bf Proof of point 1:} According to Theorem \ref{thm:partial-hausdorff-convergence}, for every $x\in T$ and $x_k \in K_k$ such that $\omegalim _k \phi_k (x_k) = F_1 (x)$, the orbit $R_k (\gamma) \phi _k (x_k)$ converges to $R_\infty (\gamma) F_1 (x)$. Choose $x=o$ and recall that we chose $F_1 (o) =o$. Choose the sequence $x_k:=o_k\in K_k$; by construction the $\omega$-limit of $\phi _k (o_k)$ is $o$. For any $\gamma\in \Gamma$, we have $$\omegalim R_k(\gamma)\cdot o = \omegalim R_k(\gamma)\cdot \phi_k(o_k) = R_\infty(\gamma)\cdot o.$$ Therefore the sequence of orbit maps $\tau_{R_k}$ converges to $\tau _{R_\infty}$ along the ultrafilter.

{\bf Proof of point 2:} It is a direct consequence of the fact that, by assumption, $\rho _\infty$ is a Schottky representation on $T$, and of the Proposition \ref{QI2}.
\end{proof}

We have now completed the proof of \Cref{hausdimasymp} and our main objective in this paper. 

We exhibit in the next section an actual subsequence along which $R_k$ converge to $R_\infty$ in $\Repr_{\H^\infty}(\F_r)$.

\subsection{Convergent subsequences of representations}\label{sec:convergent_subsequences}

Let us come back to the end of the proof of \Cref{thm:partial-hausdorff-convergence}, in the case of a general finitely generated group $\Gamma$. We use freely in this section the notations of Sections \ref{sec:lifts} and \ref{sec:approx}. The last part of the proof of \Cref{thm:partial-hausdorff-convergence} actually proves the following: for each $l\in \N$, $\gamma\in \Gamma$, for all $k$ such that $\ell(k)\geq l$ and $\gamma \in \Gamma_{\ell(k)}$, we have:
\begin{equation}\label{explisub}
d(R_k(\gamma)\cdot F_1(x^l), R_\infty(\gamma)\cdot F_1(x^l)) \leq 2 \epsilon_{\ell(k)}.
\end{equation}
This explicit estimates leads to the following construction of actual convergent subsequences of representations:
\begin{proposition}[Description of convergent subsequences]\label{prop:convergent_subsequences}
Let $J\subset \N$ be a subset such that $\lim_{k\in J} \ell(k)=+\infty$. Then we have:
\begin{itemize}
    \item The subsequence $(R_k)_{k\in J}$ converges to $R_\infty$ in $\Repr_{\H^\infty}(\Gamma)$;
    \item if moreover $\Gamma=\F_r$ and $R_\infty$ is Schottky, we have: $\lim_{k\in J} \Hdim \Lambda_{R_k} = \Lambda_{R_\infty}$.
\end{itemize}
\end{proposition}

\begin{proof}
The second point follows from the first one and \Cref{coro:BowenAlgConv}.

For the first point, we know that for each $\gamma\in \Gamma$, for each $x\in E$, the sequence $R_k(\gamma)\cdot F_1(x)$, for $k\in J$ converges to $R_\infty(\gamma)\cdot F_1(x)$ by \Cref{explisub}.
Note that $E$ is dense in $T$ and $F_1(T)$ is total in $\H^\infty$, so $F_1(E)$ is also total. Let us prove that this implies the first point.

If two elements of $\Isom\, \H^\infty$ coincide on a total subspace, they are equal. In other terms, the topology on $\Isom \H^\infty$ defined by the simple convergence on $F_1(E)$ is Hausdorff. Then, by minimality of the Polish topology on $\Isom\, \H^\infty$ \cite[Corollary 1.7]{Duchesne}, the sequence $R_k(\gamma)$ converges to $R_\infty(\gamma)$ in $\Isom\, \H^\infty$ iff for all $z\in F_1(E)$, the sequence $R_k(\gamma)\cdot z$ converges to $R_\infty(\gamma)\cdot z$ in $\H^\infty$. This last condition is granted by Eq \eqref{explisub} along any subset $J$ such that $\lim_{k\in J} \ell(k) = +\infty$.
\end{proof}
Beware that in the Proposition, we cannot guarantee that there is such a subset $J$ which belongs to the ultrafilter. However, the description  is explicit enough to be actually used in examples.

\begin{example}[McMullen example, continued II]\label{ex:mcmullen3}
 In the case of \Cref{ex:mcmullen}, one can explicitly lift the limiting trivalent tree into the renormalized spaces $X_\theta := (\H^2,r_\theta^{-1}d,o)$ by taking the orbit of the origin by the groups generated by the three symmetries, and joining any neighbors by a geodesic segment.

For each $\theta$, this lift is isometric on each edge, and ever closer to an isometric embedding. Once all the choices made, and with straightforward notations, it translates in our setting in the fact that $\ell(\theta)\xrightarrow{\theta\to 0} +\infty$. In other terms, in this case, we do not need ultrafilter and directly get the limit:
$$\omegalim 4 |\log (\theta /2)| \Hdim \Lambda _{\rho _\theta} =
\Hdim \Lambda _{\rho _\infty} = 2\log 2.$$
This recovers \Cref{McMullenexample}.
\end{example}

\bibliographystyle{alpha}
\bibliography{biblio}

\end{document}